\documentclass{llncs}
\usepackage{amsmath}
\usepackage{amssymb}
\usepackage{enumerate}
\usepackage{varioref}
\usepackage{charter,eulervm}
\usepackage{graphicx}
\usepackage{boxedminipage}
\usepackage{colortbl, subfigure}

\newcommand{\es}{\varnothing}

\DeclareMathOperator{\PP}{PP}
\DeclareMathOperator{\FLS}{FLS}
\DeclareMathOperator{\NP}{N-P}
\DeclareMathOperator{\SP}{SP}
\DeclareMathOperator{\EGP}{EGP}

\title{{\sc Set Representations of Linegraphs}}

\author{
Jun-Lin Guo\inst{1}
\and
 Tao-Ming Wang\inst{2}
\and
Yue-Li~Wang\inst{1}\thanks{All correspondence should be addressed to
Professor Yue--Li Wang, Department of Information Management,
National Taiwan University of Science and Technology, Taipei,
Taiwan (Email: {\tt ylwang@cs.ntust.edu.tw}).}
\and
 Ton~Kloks\inst{3}
}
\institute{
 Department of Information Management\\
 National Taiwan University of Science and Technology\\
 {\tt ylwang@cs.ntust.edu.tw}
\and
Department of Mathematics\\
Tunghai University, Taichung, Taiwan\\
\and
Department of Computer Science\\
 National Tsing Hua University, Taiwan\\
}

\begin{document}

\maketitle

\begin{abstract}
Let $G$ be a graph with vertex set $V(G)$ and edge set $E(G)$. A
family $\mathcal{S}$ of nonempty sets $\{S_1,\ldots,S_n\}$ is a set
representation of $G$ if there exists a one-to-one correspondence
between the vertices $v_1, \ldots, v_n$ in $V(G)$ and the sets in
$\mathcal{S}$ such that $v_iv_j \in E(G)$ if and only if $S_i\cap
S_j\neq \es$. A set representation $\mathcal{S}$ is a distinct
(respectively, antichain, uniform and simple) set representation if
any two sets $S_i$ and $S_j$ in $\mathcal{S}$ have the property
$S_i\neq S_j$ (respectively, $S_i\nsubseteq S_j$, $|S_i|=|S_j|$ and
$|S_i\cap S_j|\leqslant 1$). Let $U(\mathcal{S})=\bigcup_{i=1}^n
S_i$. Two set representations $\mathcal{S}$ and $\mathcal{S}'$ are
isomorphic if $\mathcal{S}'$ can be obtained from $\mathcal{S}$ by a
bijection from $U(\mathcal{S})$ to $U(\mathcal{S}')$. Let
$F$ denote a class of set representations of a graph $G$. The
type of $F$ is the number of equivalence classes
under the isomorphism relation. In this paper, we investigate
types of set representations for
linegraphs. We determine the types for the following categories of set
representations: simple-distinct, simple-antichain, simple-uniform
and simple-distinct-uniform. \vskip 0.4in


\end{abstract}

\section{Introduction}

Let $G$ be a simple graph with vertex set $V(G)$ and edge set
$E(G)$. The degree of $v$ is denoted by $d_G(v)$. For brevity,
$V(G)$, $E(G)$ and $d_G(v)$ are simply written as $V$, $E$ and
$d(v)$ when the context is clear.

\begin{definition}
Let $\mathcal{S}$ be a multiset of nonempty sets $\{S_1, \dots,
S_p\}$, that is,
$S_1, \ldots, S_p$ might not be distinct.
We write
\[U(\mathcal{S})=\bigcup_{i=1}^p S_i, \quad\text{and we
call $U$ the universe of $\mathcal{S}$.}\]
The {\em intersection graph\/} of $\mathcal{S}$
is the graph $G(\mathcal{S})=(V,E)$ with
\[V=\{\;S_1,\; \ldots,
\;S_p\;\} \quad\text{and}\quad E=\{\;(S_i,S_j):\;
i\neq j \quad\text{and}\quad S_i\cap S_j\neq \es\;\}.\]
\end{definition}
We say that $\mathcal{S}$ is a {\em set representation\/} of the the
intersection graph $G$. We write $F(G)$ for the set of all set
representations of $G$. The smallest cardinality of a universe,
i.e., $|U(\mathcal{S})|$, for which $G$ has a set representation is
called the intersection number of $G$ and it is denoted by
$\theta(G)$.

\bigskip

We distinguish the following categories of set representations
$\mathcal{S}=\{S_1,\dots,S_p\}$.
\begin{description}
\item[Distinct] if no two sets $S_i$ and $S_j$ in $\mathcal{S}$
are the same.
\item[Antichain] if no set $S_i$ is a subset of another set $S_j$.
\item[Uniform] if all subsets $S_i$ have the same cardinality.
\item[Simple] if any two subsets have at most one element in common.
\end{description}

\bigskip

Below follows some notation and terminology.
\begin{enumerate}[\rm 1.]
\item The intersection
numbers of distinct, antichain, uniform and simple set
representations $\mathcal{S}$ of $G$ are denoted by
\[\theta_d(G), \theta_a(G),
\theta_u(G), \text{ and } \theta_s(G) \mbox{, respectively}.\]
\item Let
$F_d(G)$, $F_a(G)$, $F_u(G)$ and
$F_s(G)$ be the sets of all minimum distinct, antichain,
uniform, and simple set representations $\mathcal{S}$
of $G$. That is,
\[|U(\mathcal{S})|=\theta_d(G) \quad\text{if}\quad
\mathcal{S} \in F_d(G),\] and similarly for the minimum
universes of the other
types.
\item We write
\begin{enumerate}[\rm (a)]
\item $F_{sd}(G)=F_s(G)\cap F_d(G)$,
\item $F_{sa}(G)=F_s(G)\cap F_a(G)$,
\item $F_{su}(G)=F_s(G)\cap F_u(G)$ and
\item $F_{sdu}(G)=F_{sd}(G)\cap F_{su}(G)$.
\end{enumerate}
\item A set
representation $\mathcal{S}$ in $F_\mathrm{sd}(G)$ is called an
sd-set representation. The sa, su and sdu-representations are
defined similarly. The minimal cardinalities of a universe
$U(\mathcal{S})$ with $\mathcal{S} \in F_\mathrm{sd}(G)$, is
denoted by $\theta_\mathrm{sd}(G)$. The parameters
$\theta_\mathrm{sa}(G)$, $\theta_\mathrm{su}(G)$ and
$\theta_\mathrm{sdu}(G)$ are defined similarly.
\end{enumerate}

\bigskip

It seems that Szpilrajn-Marczewski~\cite{szpi45} first came up with
the idea of an intersection graph and a set representation although,
often it is attributed to Erd\H{o}s, Goodman and
P\'{o}sa~\cite{erdo66}. In~\cite{kou78}, Kou, Stockmeyer and Wong
proved that the computation of $\theta(G)$ is an NP-complete
problem. Poljak, R\H{o}dl and Turz\'{\i}k proved the NP-completeness
of $\theta_\mathrm{d}(G)$ and $\theta_\mathrm{s}(G)$~\cite{polj81}.
We refer to~\cite{tsuc98} for the NP-completeness of
$\theta_\mathrm{a}(G)$ and $\theta_\mathrm{u}(G)$. Kong and Wu
investigated bounds and relations between the various categories of
set representations~\cite{kong09}. We remark that the sets
$F_{d}(G)$, $F_{a}(G)$, $F_{u}(G)$ and $F_{s}(G)$ are not empty
(see, eg,~\cite[Theorem~2.5]{hara72}).

\begin{definition}
Two set representations $\mathcal{S}$ and $\mathcal{S}'$ of
$F(G)$ are isomorphic
if there is a bijection
$U(\mathcal{S}) \rightarrow U(\mathcal{S}')$ which maps
each set of
$\mathcal{S}$ to a unique set of
$\mathcal{S}'$.
\end{definition}

\begin{definition}
A graph $G$ is {\em
uniquely intersectable\/} if all elements
of
$F(G)$ are {\em isomorphic\/}.
\end{definition}
Alter and Wang~\cite{alte77}
studied uniquely intersectable graphs.
Unique simple, distinct, and antichain intersectability
was subsequently
studied in~\cite{bylk97,maha99,tsuc90}.

\bigskip

We parameterize intersectability as follows.

\begin{definition}
\label{intersectable} Let $F$ be some category of set
representations. We say that $F$ is {\em of type $\ell$\/}  if its
members are partitioned into $\ell$ equivalence classes by the
isomorphism relation. We call $\ell$ the {\em type\/} of $F$ and we
denote it by $\tau(F)$.
\end{definition}

When $F$ is the collection of all set representations of a certain
category $\tt x$, then we also write $\tau_{\tt x}(G)$ instead of
$\tau(F)$. Thus, by definition,
\[\tau(F_s(G))=\tau_s(G)=1, \quad \tau(F_d(G))=\tau_d(G)=1
\quad\text{and}\quad
\tau(F_a(G))=\tau_a(G)=1.\]

\bigskip

The {\em linegraph\/} of a graph $G=(V,E)$ is
the graph $G^*$ with
\[V(G^*)=E \quad\text{and}\quad
E(G^*)=\{ef: \{e,f\}\subseteq E \text{ and }
e \cap f \neq \es\}.\]
Bylka and Komar \cite{bylk97} and Li and Chang \cite{li08}
investigated the characterization of graphs $G$ with
$\tau(F_{s}(G^*))=1$.

We summarize our results in this paper as follows.

\begin{enumerate}[\rm(1)]
\item If $G$ is not one the following graphs: $K_4, W_t, 3K_2\vee
K_1$, a star, or a $1t$-peacock, then
\[\theta_{sd}(G^*)=|V_i|+\sum_{i=1}^{k}m_i.\]
The set $V_i$ is the set of `inland vertices,' which we define
in Definition~\vref{df critical and inland}. The sum is the
total number of vertices of degree one that are adjacent to some
inland vertex. We prove this formula in
Theorem~\ref{LessThanGamma}.
\item If $G$ is not one of the following
graphs: $K_3$, $K_4$, $W_t$, a star, or a tailed peacock,
then
\[\theta_{sa}(G^*)=|V_i|+\sum_{i=1}^{k}(m_i+1).\]
We prove this formula in Theorem~\vref{LessThanGamma'}.
\item
\[\tau_{sd}(G^*)=
\begin{cases}
2 & \text{if $G$ is $K_4$, $W_t$, or a TP$_1$,}\\
3 & \text{if $G$ is $3K_2\vee K_1$,}\\
2+N_{PP}(d(v), r) & \text{if $G$ is a $v$-star with $d(v)\geq 3$,}\\
2^{|V_{3w}|} & \text{otherwise.}
\end{cases}\]
\item
\[\tau_{sa}(G^*)=
\begin{cases}
5 & \text{if $G=\mathrm{TP}_2$, $m_1,m_2\geq 2$, $m_1\neq m_2$,}\\
4 & \text{if $G=\mathrm{TP}_2$, $m_1=m_2\geq 2$,}\\
3 & \text{if $G=\mathrm{TP}_2$, $m_1\geq 2$, $m_2=1$,}\\
2 & \text{if $G \in \{K_4,W_t, \mathrm{TP}'\}$,}\\
1 & \text{if $G$ is $K_3$,}\\
1+N_{PP}(d(v), r) & \text{if $G$ is a $v$-star with $d(v)\geqslant 3$,}\\
2^x y^z & \text{otherwise.}
\end{cases}\]
\item
\[\tau_\mathrm{sdu}(G^*)=
\begin{cases}
2 & \text{if $G \in \{K_4,W_2,\mathrm{TP}_1\}$, $m_1=1$,}\\
N_{PP}(d(v), r) & \text{if $G=H$,
$\theta_{sdu}(K_{d(v)})=d(v)$,}\\
1+N_{PP}(d(v)+1, r) & \text{if $G=H$,
$\theta_{sdu}(K_{d(v)})=d(v)+1$,}\\
1 & \text{otherwise.}
\end{cases}\]
\end{enumerate}
Peacocks will be defined in Definition~\vref{df peacocks} and
$v$-stars in Definition~\vref{df v-stars}. Above, we denote tailed
peacocks by TP; TP$_1$ is a $1t$-peacock; and TP$_2$ is a
$2t$-peacock. The graph TP$'$ denotes the tailed peacocks TP
excluding the TP$_2$ in previous conditions. The graph $H$ denotes
$v$-stars with $d(v)\geq 4$, $x$ is the number of vertices $v_i$
with $m_i=2$ and $d(v_i)=3$, $y=3+N_{PP}(m_i+1,r)$ and $z$ is the
number of vertices $v_i$ with $m_i\geq 3$ satisfying $d(v_i)=m_i+1$.

\bigskip

This paper is organized as follows. Section~\ref{Preliminaries} has
the preliminaries. Section~\ref{intersectabilityofK_n} deals with
the types $\tau_\mathrm{sd}(K_n), \tau_\mathrm{sa}(K_n)$ and
$\tau_\mathrm{sdu}(K_n)$. Sections~\ref{thetaF_sd}
and~\ref{thetaF_sa,su,sdu} contains the derivation of the types
$\tau_\mathrm{sd}(G^*),  \tau_\mathrm{sa}(G^*)$ and
$\tau_\mathrm{sdu}(G^*)$.

\section{Preliminaries}
\label{Preliminaries}

For $n \in \mathbb{N}$, let $[n]=\{1,\ldots,n\}$.
If every pair of vertices in a
graph is adjacent then the graph is called a clique.
A clique with $n$ vertices is denoted as $K_n$.
A {\em $k$-clique\/} in
graph $G$ is an induced subgraph of $G$ which is a clique with
$k$ vertices.
{\em trivial clique\/} contains only one vertex. A
3-clique is also called a triangle.

\begin{definition}
A set $\mathcal{Q}=\{Q_1, \ldots, Q_p\}$ of cliques in $G$ is an
{\em edge-clique cover\/} of $G$ if
\[V(G)=\bigcup_{i=1}^p V(Q_i) \quad\text{and}\quad
E(G)=\bigcup_{i=1}^p E(Q_i).\] An edge-clique
cover $\mathcal{Q}$ is called an
{\em edge-clique partition\/} if $E(Q_i)\cap E(Q_j)=\es$
for all distinct $i,j\in [p]$.
\end{definition}

In Section~\ref{n clique covers}, we recall that,
by the Erd\H{o}s -- De Bruijn theorem, there exist two
kinds of nontrivial edge-clique covers $\mathcal{Q}$ of $K_n$ with
$|\mathcal{Q}|=n$. We introduce a third, trivial one
for ease of future arguments.
In Section~\ref{simple represenation}, we show
that, via a bijection introduced by Erd\H{o}s, Goodman and P\'{o}sa, one
obtains at least three kinds of nonisomorphic simple set representations
$\mathcal{S}$ of $K_n$ with $\theta_\mathrm{s}(K_n)=n$.

\subsection{Edge clique-covers $\mathcal{Q}$ of $K_n$ with $|\mathcal{Q}|=n$}
\label{n clique covers}

\begin{definition}
\label{FLS} A {\em finite linear space\/} $\Gamma$, abbreviated
$\FLS$, is a pair $\Gamma=(P,L)$, where $P$ is a set of $n$ {\em
points\/} and $L$ is a set of {\em lines\/} satisfying the following
conditions:
\begin{description}
\item[(L1)] A line is a set of at least two and at most $n-1$ points.
\item[(L2)] Any two points are on exactly one line.
\end{description}
\end{definition}

\begin{definition}
An $\FLS$ is a {\em finite projective plane\/}, abbreviated $\PP$,
if the following two conditions are satisfied:
\begin{description}
\item[(P1)] Any two distinct lines intersect in exactly one point.
\item[(P2)] There are four points of which no three are on a line.
\end{description}
\end{definition}

The following theorem is well-known, see, e.g.,~\cite{batt93}.

\begin{theorem}
\label{project plane}
If $\Pi$ is a $\PP$ with $n$ points and $l$ lines, then
there exists an integer $r$, called the order of the
plane, such that
\[n=l=r^2+r+1.\]
Furthermore, each
point lies on $r+1$ lines and each line contains
$r+1$ points.
\end{theorem}
The $\PP$ of order $2$ is the well-known \textit{Fano Plane}
(See~Figure~\ref{Fano Plane}).

\begin{figure}[htb]
\centering
\includegraphics[scale=0.3]{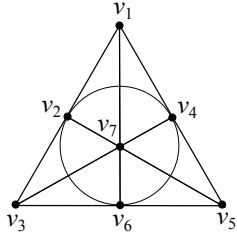}
\caption{Fano Plane.}\label{Fano Plane}
\end{figure}

It is well-known that there are unique $\PP$'s of orders 2, 3, 4, 5, 7
and 8 and none of orders 6 or 10. Moreover, there are exactly four
non-isomorphic $\PP$'s of order 9, namely the Desarguesian Plane, the
Left Nearfield Plane, the Right Nearfield Plane and the Hughes
Plane. In this paper, we use $N_\mathrm{PP}(n,r)$ to denote the
number of non-isomorphic $\PP$'s with $n$ points and order $r$.

\bigskip

Let $n \in \mathbb{N}$ and $m \in \mathbb{N}$ and assume $m > 1$.
Consider a set of $n$ elements, say $\mathcal{U}=[n]$. Let $A_1,
\ldots, A_m$ be subsets of $\mathcal{U}$ such that every pair of
elements of $\mathcal{U}$ is contained in exactly one of them. De
Bruijn and Erd\H{o}s proved the following theorem~\cite{brui48}.

\begin{theorem}
\label{deBrijnErdos}
We have that $m\geq n$. Equality occurs only in
one of the two following cases.
\begin{enumerate}[\rm (a)]
\item
\label{thmdbea}
One set, say $A_1$ contains $n-1$ elements,
say $A_1=[n-1]$. The other sets are then, without
loss of generality,
\[A_2=\{1,n\}, A_3=\{2, n\}, \ldots, A_n=\{n-1, n\}.\]
\item For some $r \in \mathbb{N}$, $n = r^2+r+1$.
All subsets have $r+1$ elements and each element of
$\mathcal{U}$ is in exactly $r+1$ subsets.
\end{enumerate}
\end{theorem}

\begin{remark}
Notice that we may assume that each subset $A_i$ has at least
two elements, otherwise we could simply remove some of them.
\end{remark}

\begin{corollary}
\label{deBr} Let $n \geq 3$ and let $\mathcal{Q}$ be an edge-clique
partition of $K_n$ with $|\mathcal{Q}|>1$. If $\mathcal{Q}$ contains
no trivial clique, then $|\mathcal{Q}|\geq n$ and equality holds
only in one of the two following cases.
\begin{enumerate}[\rm (a)]
\item
\label{(a)}
$\mathcal{Q}$ consists of one clique with
$n-1$ vertices and $n-1$ copies of $K_2$ or,
\item the $\FLS$ implied by $\mathcal{Q}$ is a $\PP$.
\end{enumerate}
\end{corollary}

The $\FLS$'s corresponding to edge-clique partitions as in
Condition~(\ref{(a)}) of Corollary~\ref{deBr} are conventionally
referred to as \textit{near-pencils}, abbreviated $\NP$. Let
$\NP(K_n)$ and $\PP(K_n)$ denote the edge-clique partitions of $K_n$
obtained from its corresponding $\NP$ and $\PP$. Table~\ref{PPandNP}
illustrates $\NP(K_7)$ and $\PP(K_7)$. Note that clique $Q_i$ of
$\PP(K_7)$ is corresponding to $\ell_i$ for $1\leq i\leq 7$ of Fano
plane as shown in Figure~\ref{Fano Plane}.

\begin{table}[htb]
\begin{center}\small
\caption{Edge clique-covers $\mathcal{Q}$ of $K_7$ with
$|\mathcal{Q}|=7$}\label{PPandNP}
\begin{tabular}{|l|l|l|l|}\hline
\multicolumn{1}{|c|}{$\mathcal{Q}$} &
\multicolumn{1}{|c|}{N-P($K_7$)} &
\multicolumn{1}{|c|}{PP($K_7$)} &
\multicolumn{1}{|c|}{SP($K_7$)} \\
\hline\hline
 $Q_1$ & $\{v_1,\ldots,v_6\}$ & $\{v_1,v_2,v_3\}$ & $\{v_1,\ldots,v_7\}$ \\
 \rowcolor[gray]{0.8}\color{black} $Q_2$ & $\{v_1, v_7\}$ & $\{v_1,v_4,v_5\}$ & $\{v_2\}$ \\
 $Q_3$ & $\{v_2,v_7\}$ & $\{v_1,v_6,v_7\}$ & $\{v_3\}$ \\
 \rowcolor[gray]{0.8}\color{black} $Q_4$ & $\{v_3,v_7\}$ & $\{v_2,v_4,v_6\}$ & $\{v_4\}$ \\
 $Q_5$ & $\{v_4,v_7\}$ & $\{v_2,v_5,v_7\}$ & $\{v_5\}$ \\
 \rowcolor[gray]{0.8}\color{black} $Q_6$ & $\{v_5,v_7\}$ & $\{v_3,v_4,v_7\}$ & $\{v_6\}$ \\
 $Q_7$ & $\{v_6,v_7\}$ & $\{v_3,v_5,v_6\}$ & $\{v_7\}$ \\
\hline
\end{tabular}
\end{center}
\end{table}

A third edge-clique cover $\mathcal{Q}$ of $K_n$ with
$|\mathcal{Q}|=n$ is defined as follows. Let
$\mathcal{Q}=\{Q_1,\ldots,Q_n\}$ with $Q_1=\{v_1,\ldots,v_n\}$ and
$Q_i=\{v_i\}$ for $i \in [n]\setminus \{1\}$. It is a trivial
partition since all edges are in $Q_1$ and the other
$Q_i$ contain no edges. We call $\mathcal{Q}$ the silly partition
and denote it by $\SP(K_n)$. We introduce it because it eases some
of the arguments.

\subsection{A bijection between set representations and edge-clique covers}
\label{simple represenation}

In \cite{erdo66}, Erd\H{o}s, Goodman and P\'{o}sa found a bijection,
called {\em $\EGP$-bijection}, between set representations and
edge-clique covers of a graph $G$, as described below.

\bigskip

Let $\mathcal{Q}=\{Q_1, \ldots, Q_p\}$ be an edge-clique cover of
graph $G$.  For every $v_i\in V$ with $i\in [n]$, let
\begin{equation}
\label{eqn0}
S_i=\{ Q_j: j \in [p] \text{ and } v_i \in Q_j\}.
\end{equation}

Obviously, $v_iv_j \in E$ if and only if $S_i\cap S_j\neq \es$. Thus
$\mathcal{S}=\{S_i|v_i \in V\} \in F(G)$. Conversely, let
$\mathcal{S}\in F(G)$. We obtain an edge-clique cover for $G$ as
follows. Let
\[\mathcal{S}=\{S_i: v_i \in V\}
\text{ and } U(\mathcal{S})=\{s_1, s_2, \dots, s_p\}.\] Define $Q_j
= \{v_i: j \in [p] \text{ and } s_j \in S_i\}$. Now for any edge $pq
\in E$,
\begin{equation}
\label{eqn1}
\{p,q\} \subseteq Q_j \mbox{ if and only if } s_j \in S_p \cap S_q.
\end{equation}
Thus
$\mathcal{Q}=\{Q_1,\ldots,Q_p\}$ covers the edges of $G$.

\bigskip

Hereafter, we call the set representation which is the $\EGP$-image
of an edge-clique cover an $\EGP$-set. Likewise, we call the clique
cover which is the $\EGP$-image of a set representation an
$\EGP$-cover. We use $\EGP(\mathcal{S})$ and $\EGP(\mathcal{Q})$ to
denote the $\EGP$-cover and $\EGP$-set, respectively. When the
$\EGP$-image of an edge-clique cover $\mathcal{Q}$ is simple we
denote it by $\EGP_s(\mathcal{Q})$.

\bigskip

When $\mathcal{Q}$ is an edge-clique partition then, by
Equation~\eqref{eqn0}, any two sets $S_x$ and $S_y$ in
$\EGP_s(\mathcal{Q})$ intersect in at most one element. That is,
$\EGP_s(\mathcal{Q}) \in F_s(G)$. Conversely, let $\mathcal{S} \in
F_s(G)$ and let $\mathcal{Q}$ be the $\EGP$-image of $\mathcal{S}$.
Let $xy \in E(G)$. By Equation~\eqref{eqn1}, there is exactly one
$s_j \in S_x \cap S_y$ and so, $\{x,y\}$ is in exactly one $Q_j \in
\mathcal{Q}$. That is, $\mathcal{Q}$ is an edge-clique partition.
Table~\ref{EGP-sets} illustrates the $\EGP$-sets from the
edge-clique covers in Table~\ref{PPandNP}.

\begin{table}[htb]
\begin{center}\small
\caption{EGP-sets of $K_7$}\label{EGP-sets}
\begin{tabular}{|l|l|l|l|}\hline
\multicolumn{1}{|c|}{$\mathcal{S}$} &
\multicolumn{1}{|c|}{EGP$_\mathrm{s}$(N-P($K_7$))} &
\multicolumn{1}{|c|}{EGP$_\mathrm{s}$(PP($K_7$))} &
\multicolumn{1}{|c|}{EGP$_\mathrm{s}$(SP($K_7$))} \\
\hline\hline
 $S_1$ & $\{Q_1,Q_2\}$ & $\{Q_1,Q_2,Q_3\}$ & $\{Q_1\}$ \\
 \rowcolor[gray]{0.8}\color{black} $S_2$ & $\{Q_1,Q_3\}$ & $\{Q_1,Q_4,Q_5\}$ & $\{Q_1,Q_2\}$ \\
 $S_3$ & $\{Q_1,Q_4\}$ & $\{Q_1,Q_6,Q_7\}$ & $\{Q_1,Q_3\}$ \\
 \rowcolor[gray]{0.8}\color{black} $S_4$ & $\{Q_1,Q_5\}$ & $\{Q_2,Q_4,Q_6\}$ & $\{Q_1,Q_4\}$ \\
 $S_5$ & $\{Q_1,Q_6\}$ & $\{Q_2,Q_5,Q_7\}$ & $\{Q_1,Q_5\}$ \\
 \rowcolor[gray]{0.8}\color{black} $S_6$ & $\{Q_1,Q_7\}$ & $\{Q_3,Q_4,Q_7\}$ & $\{Q_1,Q_6\}$ \\
 $S_7$ & $\{Q_2,\ldots,Q_7\}$ & $\{Q_3,Q_5,Q_6\}$ & $\{Q_1,Q_7\}$ \\
\hline
\end{tabular}
\end{center}
\end{table}

\begin{lemma}
For $n \geq 3$, the set representations $\EGP(\NP(K_n))$,
$\EGP(\PP(K_n))$ and $\EGP(\SP(K_n))$ are all in $F_{sd}(K_n)$.
\end{lemma}
\begin{proof}
First consider $\EGP(\SP(K_n))$.
It is easy to verify that, for all $i \neq j$,
\[S_i \neq S_j \quad\text{and}\quad |S_i \cap S_j|=1.\]
Thus
$\EGP(\SP(K_n)) \in F_{sd}(K_n)$.

Since $\NP(K_n)$ and $\PP(K_n)$ are edge-clique partitions of $K_n$,
the set representations $\EGP(\NP(K_n))$ and $\EGP(\PP(K_n))$ are in
$F_s(K_n)$. The sets of $\EGP(\NP(K_n))$, constructed by
Equation~\eqref{eqn0} from the cliques in
Theorem~\vref{deBrijnErdos}~(\ref{thmdbea}), are
\[S_1=\{Q_1, Q_2\}, S_2=\{Q_1, Q_3\}, \ldots,
S_{n-1}=\{Q_1, Q_n\}, S_n=\{Q_2,\ldots,Q_n\}.\] No two sets are the
same and every pair intersect in one element. This proves
$\EGP(\NP(K_n)) \in F_{sd}(K_n)$.

Now consider the edge-clique partition $\PP(K_n)$. For any vertex
$x$, all the lines that contain $x$ intersect only in $x$. By
Equation~\eqref{eqn0} and the fact that $n \geq 3$, this implies
that $|S_x \cap S_y|=1$ and $S_x \neq S_y$ whenever $x \neq y$. This
proves $\EGP(\PP(K_n))\in F_{sd}(K_n)$. This completes the proof.
\qed\end{proof}

\begin{definition}
An {\em h-punctured\/} $\PP$ is a $\PP$ with $h$ points deleted.
\end{definition}
We denote an edge-clique partition derived from
an $h$-punctured $\PP$ by $\PP_h(K_n)$.
The $\EGP$-image of $\PP_h(K_n)$ is denoted by
$\EGP_{s}(\PP_h(K_{n-p}))$.

\begin{example}
Consider the
2-punctured $\PP$ of the Fano plane in Figure~\ref{Fano Plane} with
$v_6$ and $v_7$ removed. It contains the following
lines.
\begin{align*}
& \ell_1=\{v_1,v_2,v_3\} \quad &&
\ell_2=\{v_1,v_4,v_5\} \quad && \ell_3=\{v_2,v_4\} \\
& \ell_4=\{v_2,v_5\} \quad && \ell_5=\{v_3,v_4\} \quad && \ell_6=\{v_3,v_5\}.
\end{align*}
For the edge-clique cover we define $Q_i=\ell_i$ for $1\leq i\leq
6$. Accordingly, the $\EGP$-bijection gives
$\EGP_s(\PP_2(K_5))=\{S_1,\dots,S_5\}$ (see~Figure~\ref{FLS with 5
points 6 lines}), where
\begin{align*}
& S_1=\{Q_1,Q_2\}\quad &&
S_2=\{Q_1,Q_3,Q_4\} \quad &&  S_3=\{Q_1,Q_3\}\\
& S_4=\{Q_2,Q_4\} \quad &&
S_5=\{Q_3,Q_4\}.
\end{align*}

\begin{figure}[htb]
\centering
\includegraphics[scale=0.25]{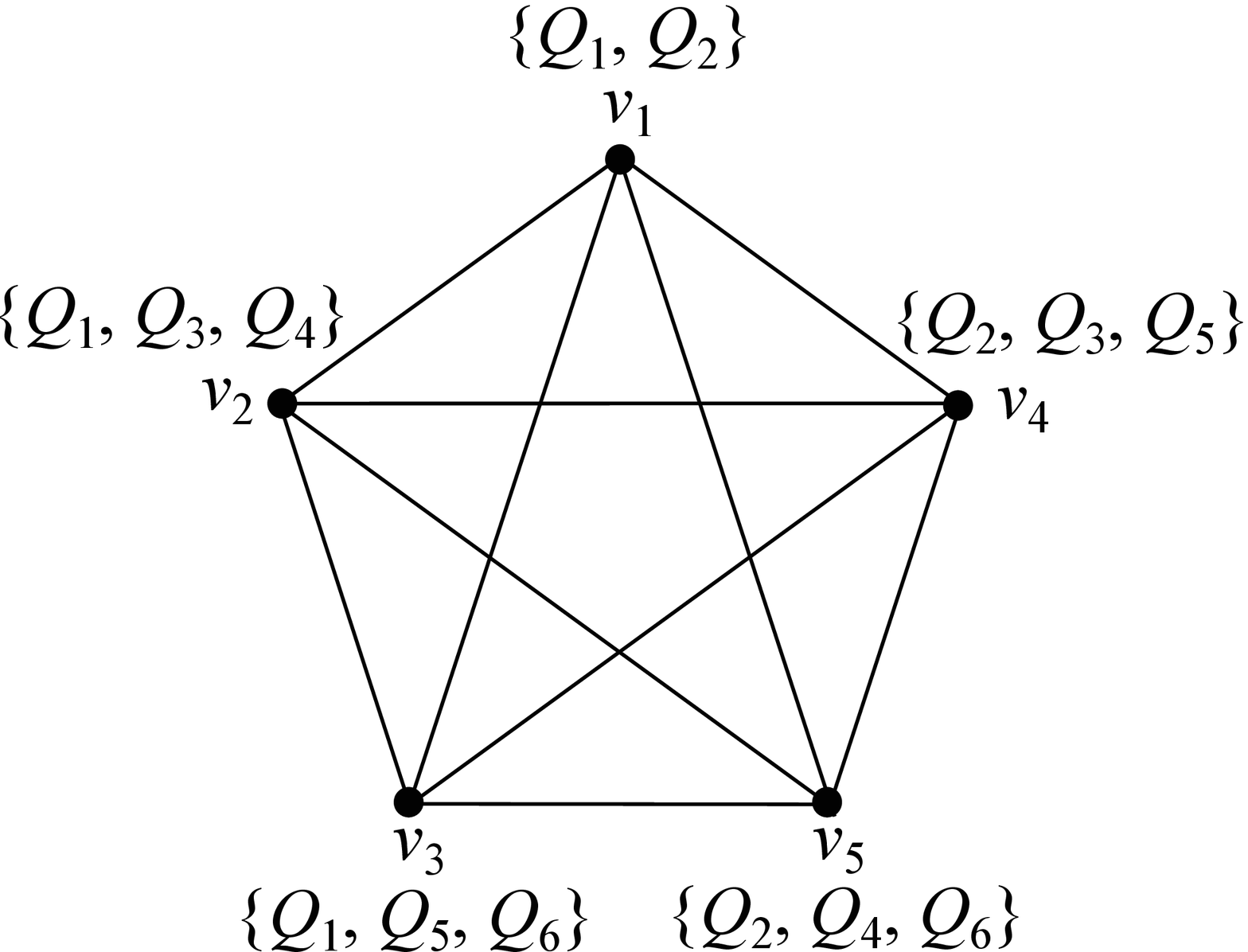}
\caption{$\EGP_{s}(\PP_2(K_5))$.}
\label{FLS with 5 points 6 lines}
\end{figure}
\end{example}

\bigskip

For a proof of the following theorem we refer to~\cite{brid72}.

\begin{theorem}
\label{bridges} Let $\Gamma=(P,L)$ be an $\FLS$ with $n$ points and
$\ell$ lines. The equality $\ell=n+1$ holds if and only if $\Gamma$
is either a $1$-punctured $\PP$ or the $2$-punctured Fano Plane.
\end{theorem}

\section{The types $\tau_{sd}(K_n)$,
$\tau_{sa}(K_n)$ and
$\tau_{sdu}(K_n)$}
\label{intersectabilityofK_n}

In this section, we investigate the types of $F_\mathrm{sd}(K_n)$,
$F_\mathrm{sa}(K_n)$ and $F_\mathrm{sdu}(K_n)$. In \cite{guo10},
Guo, Wang and Wang proved the following theorem.

\begin{theorem}
\label{SDomega K_n}
For $n\geq 1$, $\theta_{sd}(K_n)=n$ and for $n \geq 3$,
$\tau_{sd}(K_n)=2+N_{\PP}(n,r)$.
\end{theorem}

\begin{remark}
\label{remark SD} In Theorem~\ref{SDomega K_n}, the equation
$\tau_{sd}(K_n)=2+N_{\PP}(n,r)$ is due to the fact that every member
in $F_{sd}(K_n)$ for $n\geq 3$ is either $\EGP_{s}(\NP(K_n))$ or
$\EGP_{s}(\PP(K_n))$ or $\EGP_{s}(\SP(K_n))$.
\end{remark}

\bigskip

In Theorems~\ref{SAomega K_n} and%
~\ref{SDUomega K_n} we derive the types $\tau_{sa}(K_n)$ and
$\tau_{sdu}(K_n)$.

\begin{theorem}
\label{SAomega K_n} For $n\geq 3$, $\theta_{sa}(K_n)=n$ and
$\tau_{sa}(K_n)=1+N_\mathrm{PP}(n,r)$.
\end{theorem}
\begin{proof}
Assume that $n \geq 3$. By Theorem~\ref{SDomega K_n},
$\theta_{sd}(K_n)=n$. Notice that $F_{sa}(K_n)\subset F_{sd}(K_n)$
implies that $\theta_\mathrm{sa}(K_n)\geq \theta_{sd}(K_n)=n$.
Clearly, no set representation of $\EGP_s(\NP(K_n))$ is contained in
another one. Thus $\EGP_{s}(\NP(K_n))\subseteq F_{sa}(K_n)$. This
implies that $\theta_{sa}(K_n) \leq n$. As a consequence,
$\theta_{sa}(K_n)=n$.

Since $\theta_{sa}(K_n)=n$, by Remark~\ref{remark SD}, there are
only three kinds of set representations in $F_{sa}(K_n)$. It is
quickly verified that $\EGP_{s}(\NP(K_n))$ and $\EGP_{s}(\PP(K_n))$
are in $F_{sa}(K_n)$ but $\EGP_{s}(\SP(K_n))$ is not. Thus the
theorem follows. \qed\end{proof}

\begin{definition}
\label{df monopolist}
Let $\mathcal{S}\in F(G)$. An element of
$U(\mathcal{S})$ is a {\em monopolist\/} if it appears in
exactly one set of $\mathcal{S}$.
\end{definition}

\begin{theorem}
\label{SDUomega K_n}
\[\theta_\mathrm{sdu}(K_n)=
\begin{cases}
n & \text{if $n=3$,}\\
n & \text{if $n\geq 4$, $n=r^2+r+1$ and $N_{PP}(n,r)\neq 0$,}\\
n+1 & \mbox{otherwise.}
\end{cases}\]
\[\tau_\mathrm{sdu}(K_n)=
\begin{cases}
1 & \text{if $n=3$,}\\
N_{PP}(n,r) & \text{if $n\geq 4$ and $\theta_{sdu}(K_n)=n$,}\\
1+N_{PP}(n+1,r) & \text{if $n\geq 4$ and
$\theta_{sdu}(K_n)=n+1$.}
\end{cases}\]
\end{theorem}
\begin{proof}
We claim that $n\leq \theta_{sdu}(K_n)\leq n+1$. By
Theorem~\ref{SDomega K_n}, since $F_{sdu}(K_n)\subset F_{sd}(K_n)$,
we have $\theta_{sdu}(K_n)\geq n$. To prove that
$\theta_{sdu}(K_n)\leq n+1$, consider the following set
representation. Let $\mathcal{S}=\{S_1,\ldots, S_n\}$ where
$S_i=\{s_1,s_{i+1}\}$ for $1\leq i\leq n$. Clearly, $\mathcal{S}$ is
simple, distinct and uniform. Also, $|U(\mathcal{S})|=n+1$. This
implies that $\theta_{sdu}(K_n)\leqslant n+1$. This proves the
claim.

\medskip

We first analyze the cases where $\theta_{sdu}(K_n)=n$ for some
$n\geq 3$. Notice that $\EGP_s(\SP(K_n))$ is not uniform for $n \geq
3$. We also have that $\EGP_s(\NP(K_n))$ is in $F_{sdu}(K_n)$ only
for $n=3$. In a $\PP$ of order $r$, each point $\PP$ is on exactly
$r+1$ lines. This implies that the set representation of the
projective plane is in $F_{sdu}(K_n)$. Thus $F_{sdu}(K_n)$ contains
$\EGP_{s}(\NP(K_3))$ (for $n=3$) and $\EGP_{s}(\PP(K_n))$ for
$n=r^2+r+1$ ($n\geq 7$) and $N_{\PP}(n,r)\neq 0$. This shows that,
if $n\geq 4$ and there does not exist a $\PP$ of order $r$ with
$n=r^2+r+1$, then $\theta_{sdu}(K_n)=n+1$. Conversely, if $n=3$ or,
$n=r^2+r+1$ and $N_{\PP}(n,r)\ne 0$, then
\[\theta_{sdu}(K_n)=n \quad\text{and}\quad
\tau_{sdu}(K_n)=
\begin{cases}
1 & \quad\text{if $n=3$} \\
N_{\PP}(n,r)& \quad\text{otherwise.}
\end{cases}\]

\medskip

\noindent Let $\mathcal{S} \in F_{sdu}(K_n)$, $n\geq 4$, and assume
that $|U(\mathcal{S})|=n+1$.  We claim that $\mathcal{S}$ is either
$\EGP_{s}(\PP_1(K_n))$ or $\mathcal{S}=\{S_1,S_2,\ldots,S_n\}$ with
$S_i=\{s_1,s_{i+1}\}$ for $i \in [n]$. Delete all monopolists from
the sets of $\mathcal{S}$ and let $\mathcal{S}'$ be the result.
Clearly, $\mathcal{S}'\in F_{s}(K_n)$. Furthermore,
$\EGP(\mathcal{S}')$ is an edge-clique partition of $K_n$ which
contains at most $n+1$ cliques and no trivial ones. By
Corollary~\ref{deBr} and Theorem~\ref{bridges}, either
$|\EGP(\mathcal{S}')|=1$ or $\EGP(\mathcal{S}')$ is an $\NP$, a
$\PP$, a $1$-punctured $\PP$, or a $2$-punctured Fano Plane. If
$\EGP(\mathcal{S}')$ is an $\NP$ or $\PP$, then $\mathcal{S}$ is
either an $\EGP_{s}(\NP(K_n))$ or an $\EGP_{s}(\PP(K_n))$ and then
$U(\mathcal{S}) \neq n+1$. If $\EGP(\mathcal{S}')$ is an
$\EGP_{s}(\PP_2(K_n))$ (see Figure~\ref{FLS with 5 points 6 lines}
for an illustration), then $\mathcal{S}'=\mathcal{S}$ and
$\mathcal{S}$ is not uniform. Thus either
$|\mathrm{EGP}(\mathcal{S})|=1$ or $\EGP(\mathcal{S})$ is an
$\EGP_{s}(\PP_1(K_n))$. This concludes the proof of this theorem.
\qed\end{proof}

\section{The type $\tau_{sd}(G^*)$}
\label{thetaF_sd}

In the rest of this paper, let $G$ be a connected graph and
$\mathcal{P}$ an edge-clique partition of $G^*$. Our results on
$\tau_{sd}(G^*)$, $\tau_{sa}(G^*)$ and $\tau_{sdu}(G^*)$ are based
on the results in \cite{mcgu90}. Hence, we follow most of the terms
used in \cite{mcgu90}. However, for ease of readability, we repeat
some of them as follows.

\bigskip

If $uv \in E(G)$, then we use $uv$ to denote the corresponding
vertex in $G^*$. An edge in $G^*$ with endpoints $vu$ and $vw$ is
denoted by $(vu,vw)$ and a $k$-clique in $G^*$ containing vertices
$vu_1,vu_2,\ldots ,vu_k$ is denoted by $\{vu_1,vu_2,\ldots ,vu_k\}$.
In this case, we also say that the $k$-clique in $G^*$ is induced by
the edges $vu_1,vu_2,\ldots ,vu_k$ in $G$. We also use $uvw$ to
denote a triangle when $u$, $v$ and $w$ are the vertices in the
$K_3$.

\begin{definition}
\label{df v-stars} A \textit{v-star} in $G$ is a subgraph of $G$
consisting of a set of edges incident with a common vertex $v$
$($see Figure~\ref{wingsfigure}$(a))$. Denote by $S_v^i$ a star
consisting of $i$ edges incident with $v$. A $v$-star $S_v^i$ is
{\em saturated\/} $($respectively, a {\em trivial\/}$)$ if $i=d(v)$
$($respectively, $i=1)$. A {\em $v$-wing\/} in $G$ is a $K_3$ in
which only $d_G(v)>2$, and we call $v$ the {\em stalk vertex\/} of
the wing  $($see Figure~\ref{wingsfigure}$(b))$. A {\em $3$-wing} is
a $v$-wing with $d(v)=3$. A semiwing is a $K_3$ with exactly one
vertex of degree 2 in $G$ $($see Figure~\ref{wingsfigure}$(c))$. In
a semiwing, the {\em stalk vertices\/} are the vertices with degree
greater than 2. Let $w_v(\mathcal{P})$ denote the number of
3-cliques in $\mathcal{P}$ which are induced by $v$-wings. Let
$V_\mathrm{3w}$ be the set of stalk vertices of the 3-wings in $G$.
\end{definition}

\begin{figure}[htb]
\begin{center}
\subfigure[a $v$-star]{
\includegraphics[scale=0.3]{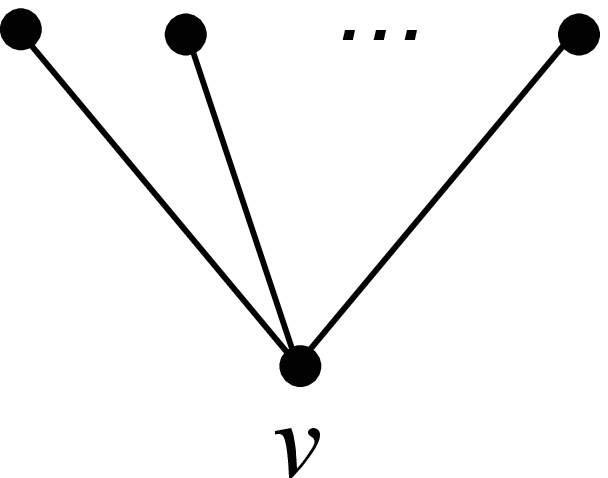}
} \quad \subfigure[two $v$-wings]{
\includegraphics[scale=0.3]{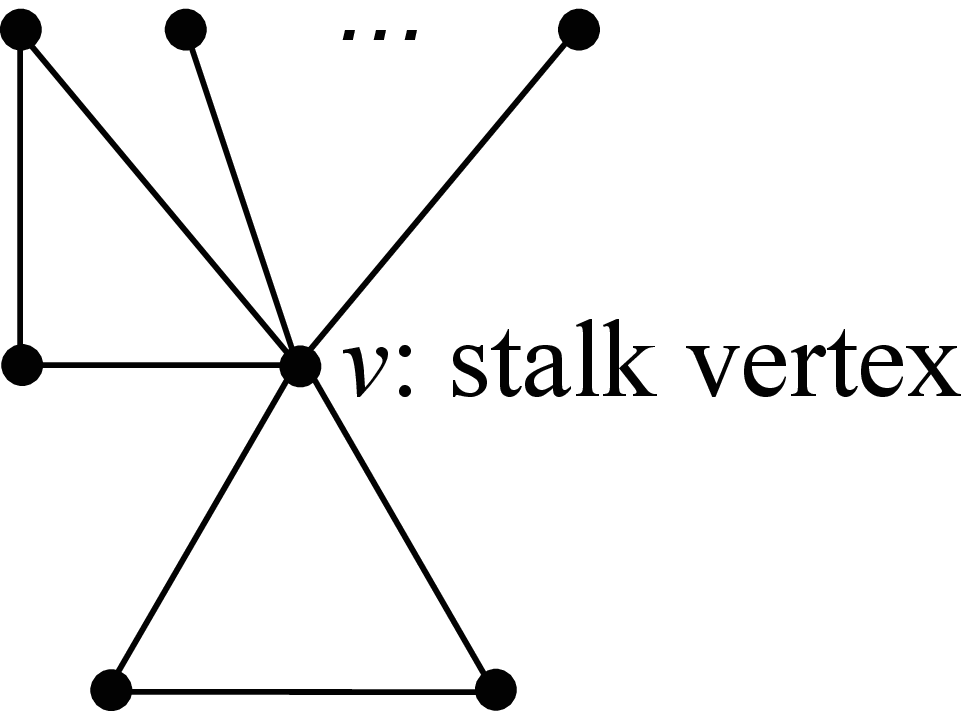}
} \quad \subfigure[a semiwing]{
\includegraphics[scale=0.3]{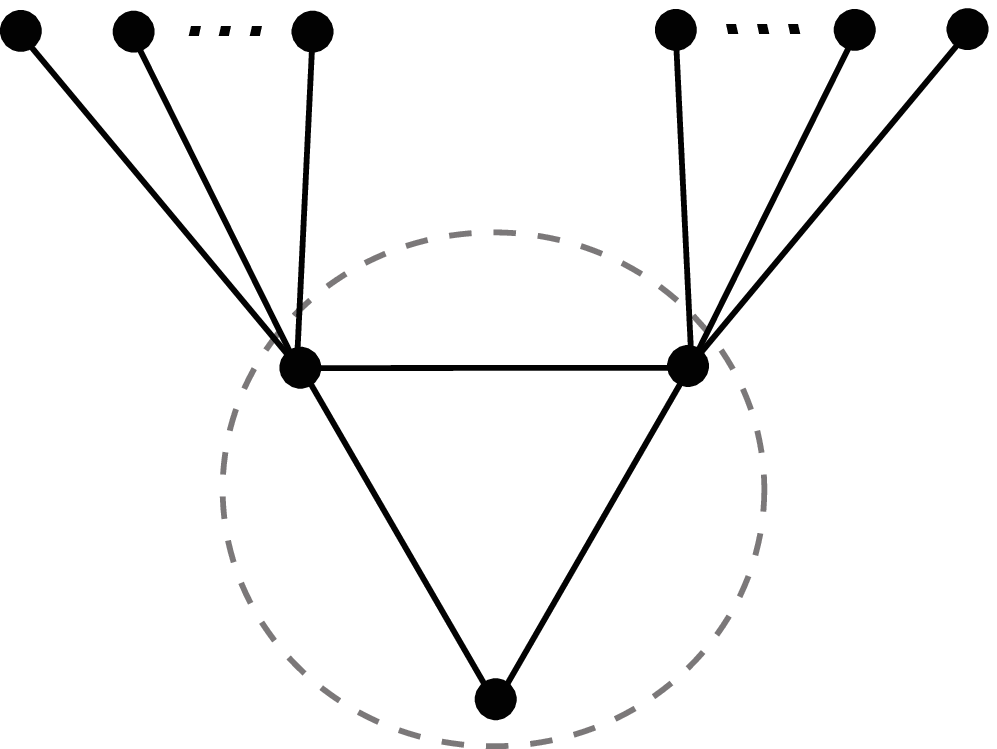}
} \quad \subfigure[$W_t$]{
\includegraphics[scale=0.3]{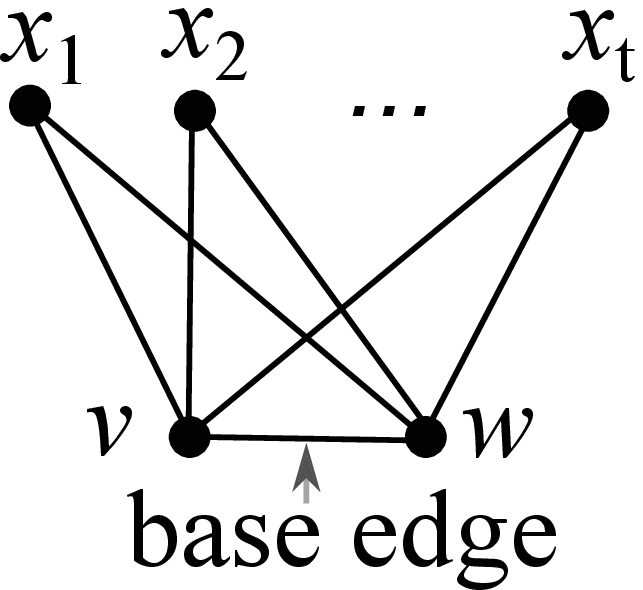}
}
\caption{Stars, wings, semiwings and $W_t$.}
\label{wingsfigure}
\end{center}
\end{figure}

The \textit{join} of graphs $G$ and $H$, denoted by $G\vee H$, is
the graph $G\vee H=(V,E)$ where $V(G\vee H)=V(G) \cup V(H)$ and
$E(G\vee H)=E(G) \cup E(H) \cup \{gh: g \in V(G)$  and $h \in
V(H)\}$. For $t\geq 2$, let $W_t=tK_1\vee K_2$ (see
Figure~\ref{wingsfigure}(d)). The {\em base edge} of $W_t$ is the
edge of which both endpoints are of degree greater than 2.

\begin{definition}
\label{df peacocks} A {\em plume} in $G$ is a vertex of degree 1. A
graph is a {\em one-tail peacock graph} if it is composed of a $K_3$
in which exactly one vertex is adjacent to $t$ plumes and the other
two vertices are of degree 2, where $t$ is a positive integer $($see
Figure~\ref{peacockgraphs}$(a))$. A graph is a {\em two-tail peacock
graph} if it composes of a $K_3$ in which exactly two vertices have
plumes as neighbors $($see Figure~\ref{peacockgraphs}$(b))$. A graph
is a {\em diamond-back one-tail peacock graph} $($respectively, {\em
diamond-back two-tail peacock graph}$)$ if it composes of a $W_t$ in
which exactly one endpoint $($respectively, both endpoints$)$ of the
base edge has plumes as neighbors $($see
Figures~\ref{peacockgraphs}(c) and ~\ref{peacockgraphs}$(d))$.
\end{definition}

For brevity, one-tail peacocks, two-tail peacocks, diamond-back
one-tail peacocks and diamond-back two-tail peacocks are abbreviated
as $1t$-peacocks, $2t$-peacocks, $d1t$-peacocks and $d2t$-peacocks,
respectively, and are denoted by TP$_1$, TP$_2$, TP$_{d1}$ and
TP$_{d2}$, respectively. We use `peacock' as a generic term for
either one of the peacock species mentioned above and denote it by
TP.

\begin{figure}[htb]
\begin{center}
\subfigure[1t-peacock]{
\includegraphics[scale=0.34]{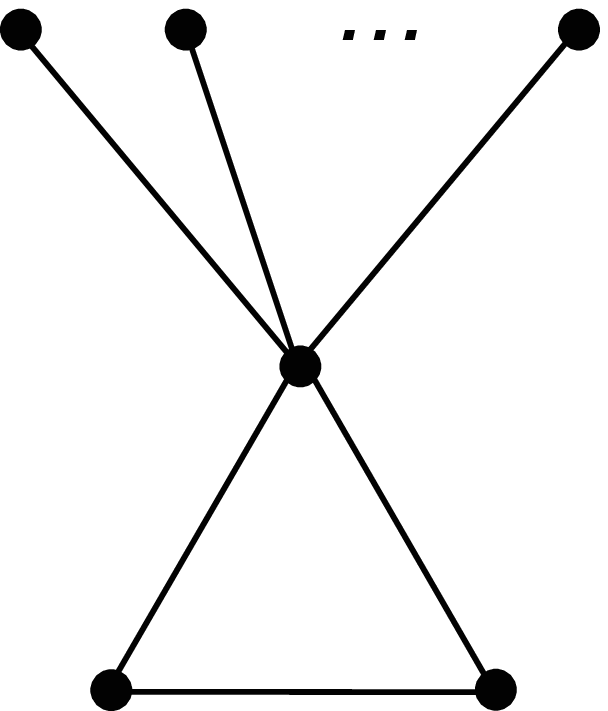}
} \quad \subfigure[2t-peacock]{
\includegraphics[scale=0.34]{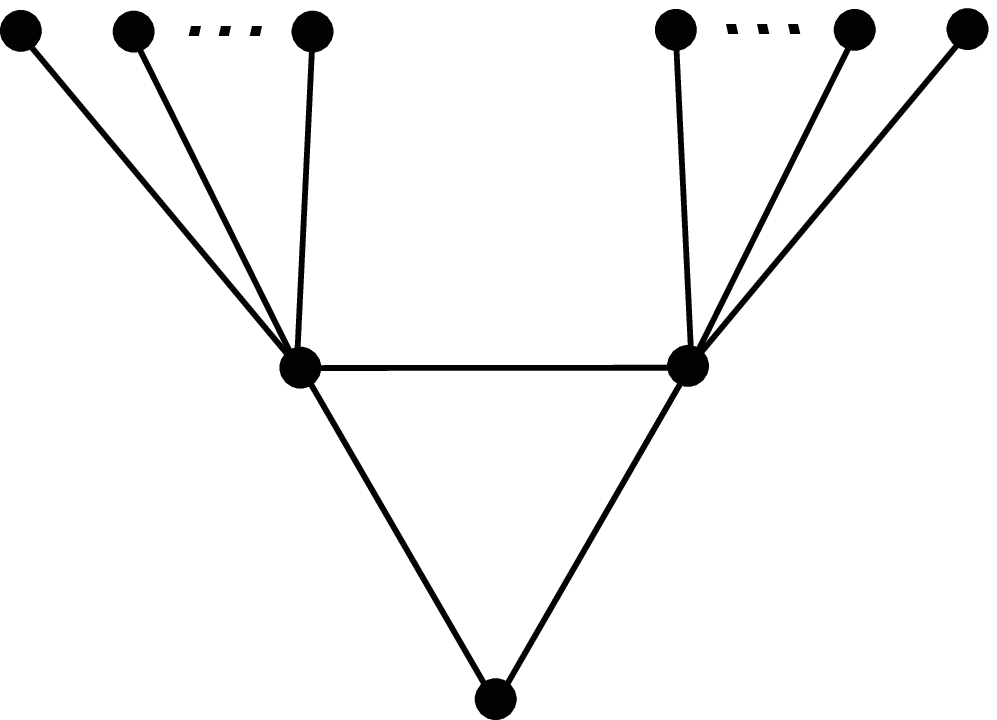}
} \quad \subfigure[d1t-peacock]{
\includegraphics[scale=0.34]{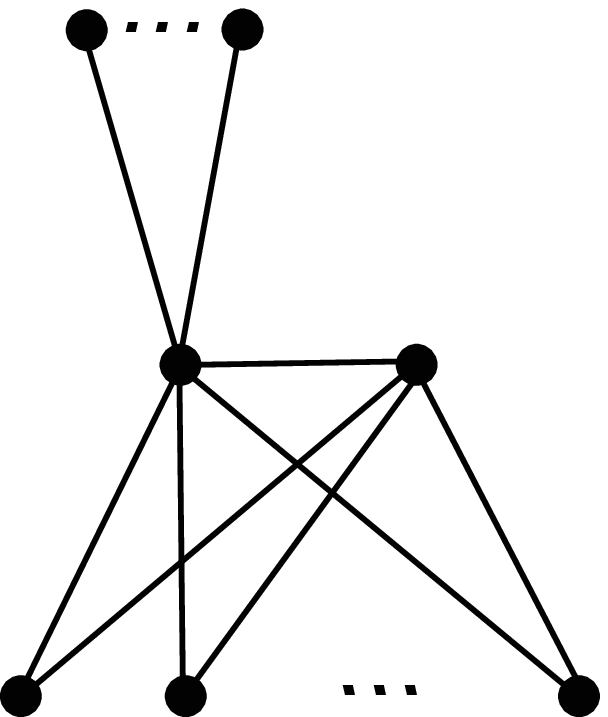}
} \quad \subfigure[d2t-peacock]{
\includegraphics[scale=0.34]{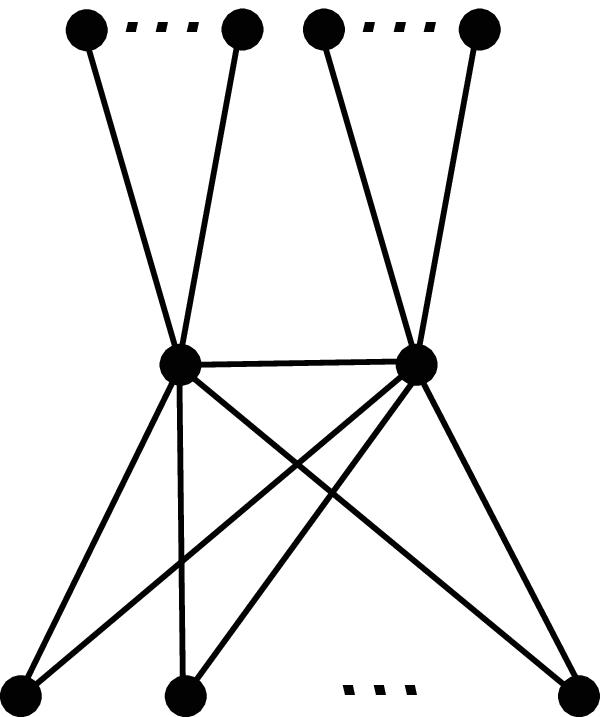}
}
\caption{tailed peacock graphs.}
\label{peacockgraphs}
\end{center}
\end{figure}

Recall Whitney's theorem~\cite{kn:whitney}.
\begin{theorem}
\label{StarOrTriangle}
Any clique in $G^*$ is induced either by a star or a $K_3$ in $G$.
\end{theorem}

\begin{lemma}[Lemma 2.13 in \cite{mcgu90}]
\label{wings}
If $w_v(\mathcal{P})> 0$ for some $v\in V(G)$, then $\mathcal{P}$
contains at least $w_v(\mathcal{P})+1$ cliques induced by $v$-stars
in $G$, with equality occurring only when $w_v(\mathcal{P})=1$ and
$d(v)=3$, or $w_v(\mathcal{P})=3$ and $G=3K_2\vee K_1$.
\end{lemma}

McGuinness and Rees \cite{mcgu90} define a surjection
$f:\mathcal{P}\rightarrow V_2(G)$ for $G\neq K_3$, where
$V_2(G)=\{v\in V(G):d(v)\geq 2\}$. We describe the surjection
as follows (and we make some modifications). Assume that $G\neq
K_3$. The set $V_2(G)$ can be partitioned into the following three
subsets with respect to $\mathcal{P}$:
\begin{enumerate}[\rm (1)]
\item $R_\mathcal{P}=\{v\in V_2(G):\text{$d(v)\geq 3$ and
no clique in $\mathcal{P}$ is induced by $v$-stars}\}$,
\item $NW_\mathcal{P}
= \{v\in V_2(G)\setminus R_\mathcal{P}:\text{$v$ does not lies in a
wing of $G$}\}$, and
\item $W = \{v\in V_2(G):\text{$v$ lies in a wing of $G$}\}$.
\end{enumerate}

\noindent By definition, $NW_\mathcal{P}$ is disjoint with
$R_\mathcal{P}$ and $W$. Whitney's theorem and
Lemma~\ref{wings} imply that $R_\mathcal{P}$ and $W$
are disjoint.

\begin{example}
Consider the graph $G$ in Figure~\ref{mapping}(a). It has 9 vertices
and 13 edges. The vertex set $V$ is $\{a,b,\ldots,i\}$ and the edges
are labeled by $1, \ldots, 13$. Figure~\ref{mapping}(b) shows the
linegraph $G^*$ of $G$. Figure~\ref{mapping}(c) is an edge-clique
partition $\mathcal{P}$ of $G^*$, where $\mathcal{P}=\{Q_1,\ldots,
Q_{17}\}$ and $Q_1=\{1,2\}$, $Q_2=\{1,5,7,8,9,10\}$, etc., as shown
in Figure~\ref{mapping}(c). Clearly, $W = \{c,d,e,f,g\}$ since $cde$
and $cfg$ are wings. Cliques $Q_{16}$ and $Q_{17}$ are induced by an
$h$-star and an $i$-star, respectively. Thus $NW_\mathcal{P}=\{h,
i\}$. Note that $d_G(b)=3$ and no clique in $\mathcal{P}$ is induced
by $b$-star. Thus $R_\mathcal{P}=\{b\}$. The partition of $V_2(G)$
is shown as in Figure~\ref{mapping}(d).
\end{example}

\begin{figure}[htb]
\begin{center}
\subfigure[a graph $G$]{
\includegraphics[scale=0.35]{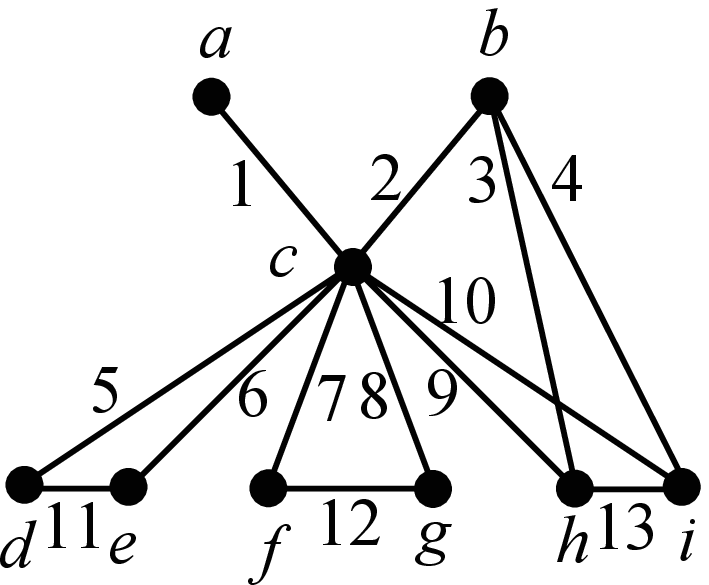}
} \quad \subfigure[line graph $G^*$]{
\includegraphics[scale=0.16]{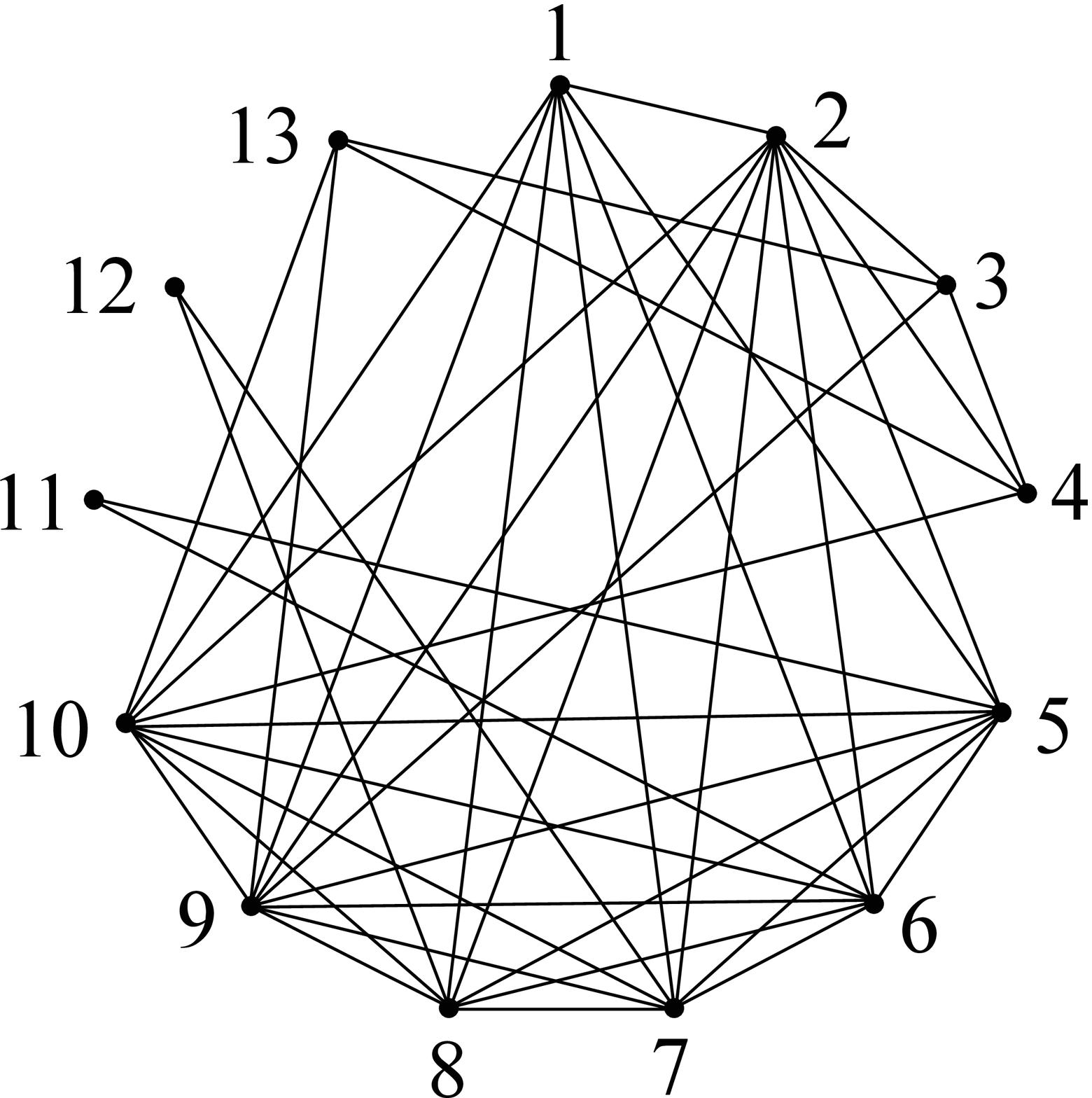}
} \quad \subfigure[a partition $\mathcal{P}$ of $G^*$]{
\includegraphics[scale=0.35]{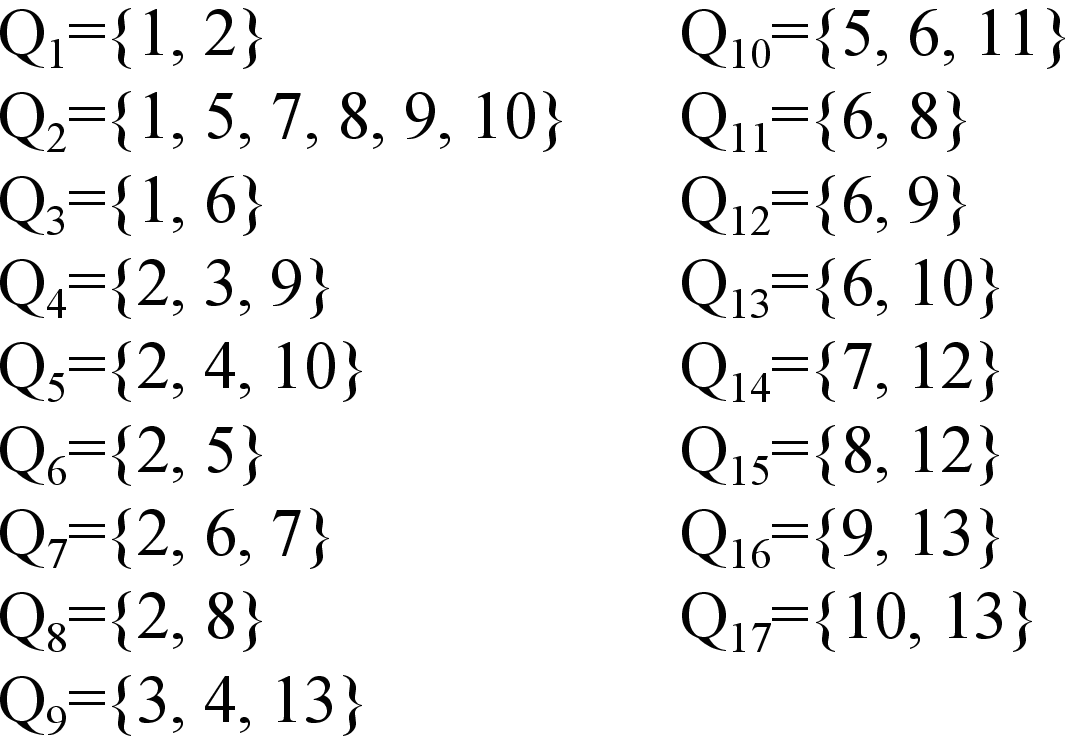}
} \quad \subfigure[Partitions of $V_2(G)$ and $\mathcal{P}$]{
\includegraphics[scale=0.25]{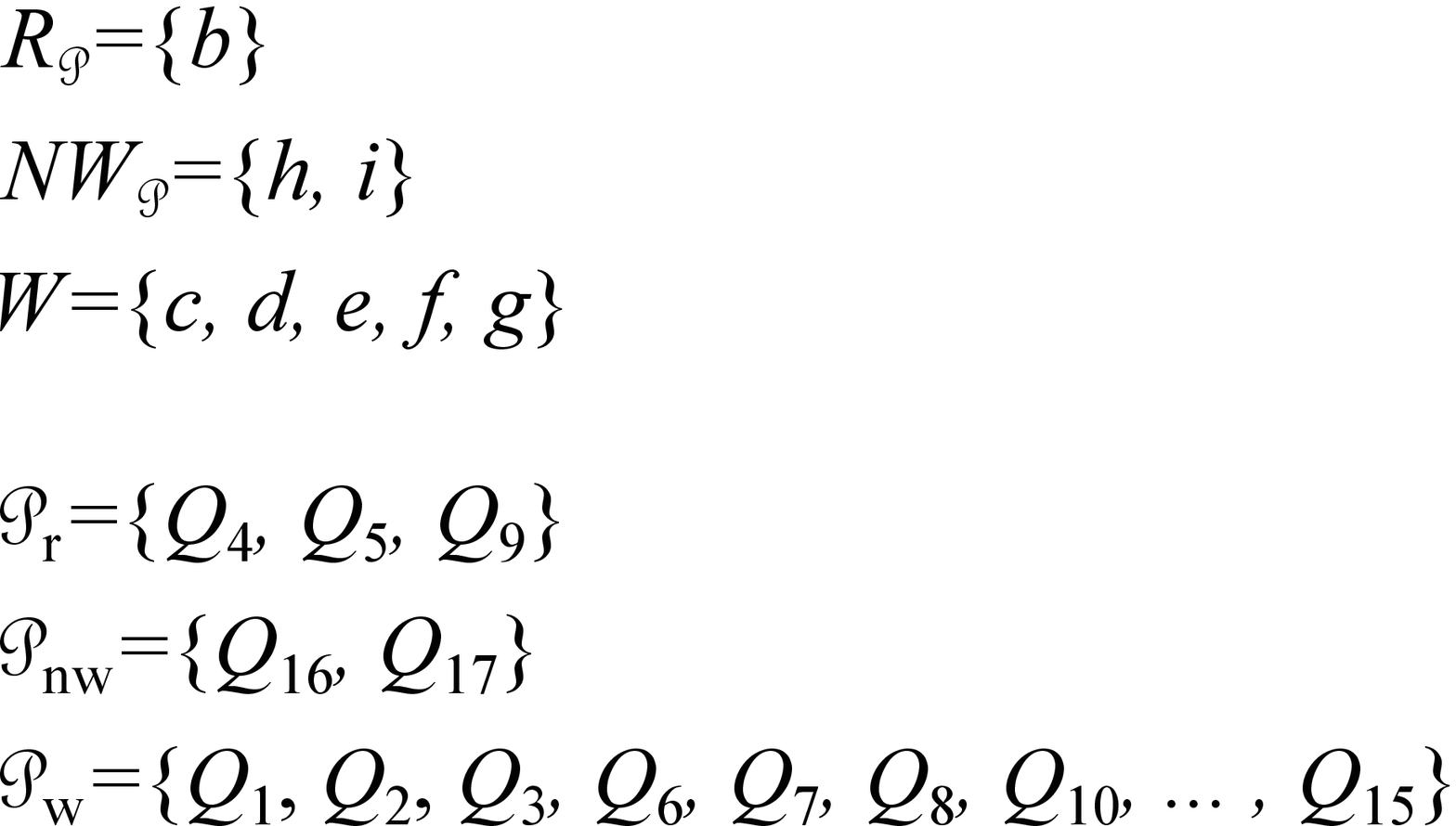}
} \quad \subfigure[$C_v$]{
\includegraphics[scale=0.25]{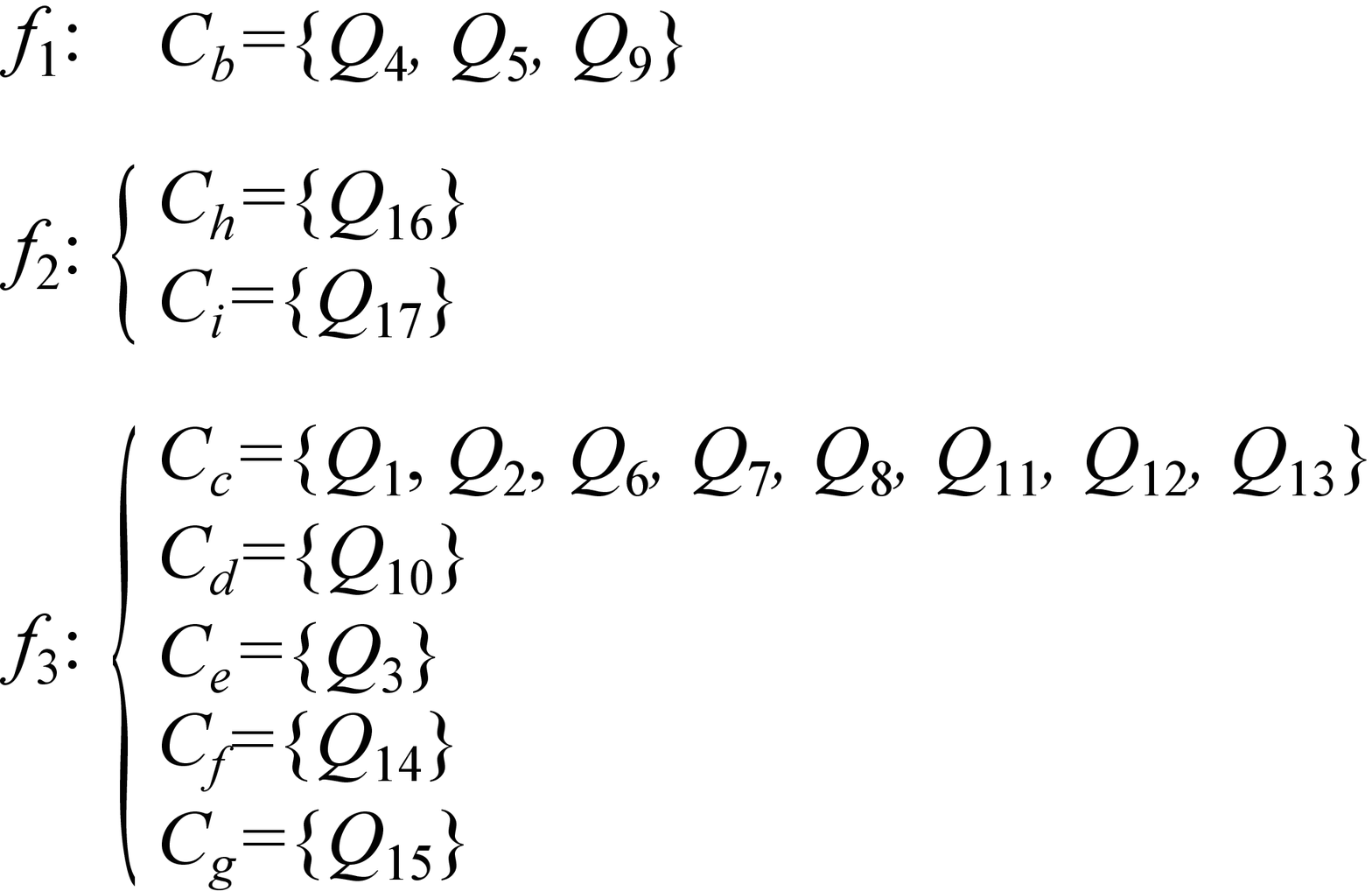}
}
\caption{$R_\mathcal{P}, NW_\mathcal{P}$ and $W$.}
\label{mapping}
\end{center}
\end{figure}

\bigskip

\begin{proposition}
\label{TypeIIIlowerbound} Let $\mathcal{P}$ be an edge-clique
partition of $G^*$ and $v$ a stalk vertex in $W$. If there are $t$
$v$-wings, then there are at least $2t+1$ cliques in $\mathcal{P}$.
\end{proposition}
\begin{proof}
For the stalk vertex $v\in W$, let
\[\{vx_{2i-1}x_{2i}:
i \in [t]\}\] be the collection of $v$-wings in $G$. We bound the
cliques of $\mathcal{P}$ from below by partitioning them into the
following types. First, by definition, there are $w_v(\mathcal{P})$
$v$-wings that are 3-cliques in $\mathcal{P}$. Secondly, by
Whitney's theorem and Lemma~\ref{wings}, there are
$w_v(\mathcal{P})+1$ cliques in $\mathcal{P}$ induced by $v$-stars.
Thirdly, for each $v$-wing $vx_{2i-1}x_{2i}$ for
$w_v(\mathcal{P})+1\leq i\leq t$, there exist two cliques in
$\mathcal{P}$ induced by $S_{x_{2i-1}}^2$ and $S_{x_{2i}}^2$.

\medskip

\noindent
Adding them up, we find that each stalk
vertex $v\in W$ gives rise to at least
\begin{equation*}
w_v(\mathcal{P})+(w_v(\mathcal{P})+1)+2(t-w_v(\mathcal{P}))=2t+1
\end{equation*}
cliques in $\mathcal{P}$. This completes the proof. \qed\end{proof}

\bigskip

Partition the nontrivial cliques $C \in \mathcal{P}$ in
the following subsets.
\begin{enumerate}[\rm I.]
\item
\label{P_r}
$\mathcal{P}_{r}=
\{C: \text{$C$ is induced by a $K_3$ in $G$ that is neither a wing
nor a semiwing}\}$,
\item $\mathcal{P}_{nw}=
\{C:\text{$C$ is induced by a semiwing or a $v$-star with
$v\in NW_\mathcal{P}$}\}$ and
\item $\mathcal{P}_{w}=\{C:\text{$C$ is induced by a $v$-star or a
$v$-wing for $v\in W$}\}$.
\end{enumerate}
Notice that
the sets $\mathcal{P}_{r}, \mathcal{P}_{nw}$ and
$\mathcal{P}_{w}$ are disjoint.

\begin{example}
We use Figure~\ref{mapping} to illustrate the
sets $\mathcal{P}_{r}$,
$\mathcal{P}_{nw}$ and $\mathcal{P}_{w}$. In
Figure~\ref{mapping}(a), we can find that triangles
$bch, bhi, bci$
and $chi$ are neither wings nor semiwings. Cliques $Q_4, Q_5$ and
$Q_9$ in $\mathcal{P}$ are induced by triangles $bch,  bci$ and
$bhi$, respectively, in $G$. Thus $\mathcal{P}_{r}=\{Q_4,
Q_5,Q_9\}$ (see Figure~\ref{mapping}(d)).

\medskip

\noindent
It is easy to verify that
\[\mathcal{P}_{w}=\{Q_1,Q_2,Q_3,\;Q_6,Q_7,Q_8,
Q_{10},\ldots, Q_{15}\}\]
in which
\begin{enumerate}[\rm (i)]
\item $Q_{10}$ is induced by a
$c$-wing,
\item $Q_{14}$ is induced by an $f$-star,
\item $Q_{15}$ is induced by
a $g$-star, and
\item all other cliques are induced by $c$-stars.
\end{enumerate}

\medskip

\noindent
The
remaining collection is
$\mathcal{P}_{nw}=\{Q_{16}, Q_{17}\}$. Note that
$Q_{16}$ is an $h$-star and $Q_{17}$ is an $i$-star.
\end{example}

\bigskip

The following lemma appears in~\cite[Lemma 2.8]{mcgu90}.
\begin{lemma}
\label{Q >= R}
$|\mathcal{P}_\mathrm{r}|\geqslant |R_\mathcal{P}|$, with equality
occurring only if either $R_\mathcal{P}=\es$ or $G=K_4$.
\end{lemma}

For a proof of the following lemma see~\cite[Lemma 2.9]{mcgu90}
\begin{lemma}
\label{semiwing}
If $G\neq W_t$ and $\mathcal{P}$ contains a clique induced by a
semiwing $wuv$ in $G$ with non-stalk vertex $w$, then $\mathcal{P}$
contains at least two cliques induced by $x$-stars for some $x\in
\{u, v\}$.
\end{lemma}

McGuinness and Rees define three functions in~\cite{mcgu90},
which we describe shortly.
\begin{enumerate}
\item A surjection $f_1:\mathcal{P}_{r}\rightarrow
R_{\mathcal{P}}$,
\item a bijection $f_2:\mathcal{P}_{nw}\rightarrow NW_{\mathcal{P}}$ and
\item another surjection $f_3:\mathcal{P}_{w}\rightarrow W$.
\label{surjection f_3}
\end{enumerate}

Since $|\mathcal{P}_{r}|\geq |R_\mathcal{P}|$ by Lemma~\ref{Q >= R},
$f_1$ can be designed to be surjective. For each stalk vertex $v\in
W$, by Proposition~\ref{TypeIIIlowerbound}, there exist at least
$2t+1$ cliques in $\mathcal{P}_{w}$ which can be assigned to the
$2t+1$ vertices lying in $v$-wings by $f_3$. Hence, $f_3$ is
surjective.

We impose some additional rules on the
construction of $f_3$. That is, for each $v$-wing $vxy$ in $G$, if
$vxy$ induces a $K_3$ in $\mathcal{P}_{w}$, then we always
assign the $K_3$ to $x$ or $y$; otherwise, assign the two 2-cliques
in $\mathcal{P}_{w}$ induced by $S_x^2$ and $S_y^2$ to $x$
and $y$, respectively. This completes the description of $f_3$.

\medskip

A
bijection $f_2: \mathcal{P}_{nw} \rightarrow NW_{\mathcal{P}}$
is described as follows. Let $u\in NW_{\mathcal{P}}$. If
$d(u)=2$ and $u$ lies in no triangle in $G$, then $S_u^2$ induces a
$K_2$ in $\mathcal{P}_{nw}$ which is assigned to $u$ in the
bijection. If $d(u)=2$ and $u$ lies in a triangle in $G$, then $u$
is the non-stalk vertex of a semiwing (assuming $G\neq K_3$). Thus
$\mathcal{P}_{nw}$ contains either the $K_2$ induced by
$S_u^2$ or the $K_3$ induced by the semiwing and either of them is
assigned to $u$. If $d(u)\geq 3$, then $\mathcal{P}$ contains
cliques induced by $u$-stars, which are in $\mathcal{P}_{nw}$
by definition, and these are assigned to $u$.

\medskip

To ensure that $f_2$ is
well-defined, recall that the non-stalk vertex of a semiwing $wuv$,
say $w$, must be in $NW_\mathcal{P}$. Therefore, if $wuv$ induces a
clique $C$ in $\mathcal{P}_{nw}$, then $f_2(C)=w$. Thus $f_2$
is a bijection function.

\bigskip

We obtain a surjection
$f:\mathcal{P} \rightarrow V_2(G)$ as follows.
When restricted to one of the domains
$\mathcal{P}_2$, $\mathcal{P}_{nw}$
or $\mathcal{P}_{w}$, the function is locally defined by
$f_1$, $f_2$ or $f_2$. Finally, we
assigning remaining trivial cliques in $\mathcal{P}$ arbitrarily
to vertices
in $V_2(G)$.
For $v \in V_2(G)$, let
\[C_v=\{C: C \in \mathcal{P} \text{ and } f(C)=v\}.\]
Figure~\ref{mapping}(e) is an example of $C_v$.

\bigskip

Before investigating $F_\mathrm{sd}(G^*)$, we define some terms. In
the rest of this section we assume that $\mathcal{S}\in
F_{sd}(G^*)$, $\mathcal{P}=\EGP(\mathcal{S})$ and $S(uv)$ denotes
the sets in $\mathcal{S}$ assigned to $uv\in V(G^*)$. Note that
$\mathcal{P}$ might contain some trivial cliques.

\begin{definition}
\label{df critical and inland} A vertex $v\in V_2(G)$ is
\textit{critical} if $v$ has a neighbor with degree 1; otherwise, it
is an \textit{inland} vertex. Let $V_\mathrm{c}$ and $V_\mathrm{i}$
denote the sets of critical and inland vertices.
\end{definition}

Assume that $V_\mathrm{c}=\{v_1, \dots, v_k\}$. For $i \in [k]$, let
$v_i^1, \dots ,v_i^{m_i}$ be the neighbors of $v_i$ of degree one.
Also, let $v_i^{m_i+1},\dots,v_i^{d(v_i)}$ be the remaining
neighbors of $v_i$.

\bigskip

We introduce some more notations.
\begin{align}
\label{eqn2}
&V_{cw}=V_\mathrm{c}\cap W\\
&V_{cnw}=V_\mathrm{c}\cap NW_{\mathcal{P}}\\
&\gamma=|V_\mathrm{i}|+\sum_{i=1}^{k}m_i\\
&\mathcal{P}_\mathrm{c}=\{C\in \mathcal{P}: \text{$C$ is induced by a
$v$-star with $v\in V_\mathrm{c}$}\} \quad\text{and}\\
&\mathcal{P}_\mathrm{i}=\mathcal{P}\setminus \mathcal{P}_\mathrm{c}.
\label{eqn3}
\end{align}

\bigskip

For $i \in [k]$, let
\begin{equation}
\label{eqn4}
\mathcal{S}_i=\{S(v_iv_i^1),
S(v_iv_i^2),\ldots, S(v_iv_i^{m_i})\}.
\end{equation}
Obviously, $\mathcal{S}_i \in F_{sd}(K_{m_i})$.  By
Theorem~\ref{SDomega K_n}, $|\EGP(\mathcal{S}_i)|\geq m_i$ for $i
\in [k]$.

\bigskip

\begin{example}
We use Figure~\ref{mapping} to illustrate the terminology that we
introduced above. In Figure~\ref{mapping}(a), the set of critical
vertices is $V_\mathrm{c}=\{c\}$ since $c$ is the only vertex which
has a neighbor that is of degree one. The set of inland vertices is
$V_\mathrm{i}=V_2(G)\setminus V_c=V\setminus\{a,c\}$.

\medskip

Thus as defined in Equations~\eqref{eqn2} up to~\eqref{eqn3}
and~\eqref{eqn4}:
\begin{align*}
& V_{cw}=\{c\}\\
& V_{cnw}=\es \\
& \gamma=|V_\mathrm{i}|+\sum_{i=1}^{k}m_i=7+1=8\\
& \mathcal{P}_\mathrm{c}=\{Q_1,Q_2,Q_3,Q_6,Q_7,Q_8,Q_{11},Q_{12},Q_{13}\}\\
& \mathcal{P}_\mathrm{i}=\{Q_4,Q_5,Q_9,Q_{10},Q_{14},\ldots,Q_{17}\}
\quad\text{and}\\
& \mathcal{S}_1=\{S(ca)\}.
\end{align*}
\end{example}

\bigskip

\begin{proposition}
\label{all members of Q(F_C^) are substars of Q(F)} All cliques in
$\EGP(\mathcal{S}_i)$, $i \in [k]$, are sub-cliques of some cliques
in $\mathcal{P}$ induced by $v_i$-stars. Furthermore,
$V_\mathrm{c}\cap R_\mathcal{P}=\es$ and $\mathcal{P}$ contains at
least $m_i$ cliques induced by $v_i$-stars for $i \in [k]$.
\end{proposition}

\bigskip

\begin{corollary}
\label{Q-I < V-I} If $|U(\mathcal{S})|\leq \gamma$, then
$|\mathcal{P}_\mathrm{i}|\leq |V_\mathrm{i}|$.
\end{corollary}
\begin{proof}
By Proposition~\ref{all members of Q(F_C^) are substars of Q(F)} and
$|\mathcal{P}|=|\mathcal{P}_\mathrm{i}|+|\mathcal{P}_\mathrm{c}|$,
we obtain that
\[|\mathcal{P}| \geq
|\mathcal{P}_\mathrm{i}|+\sum_{i=1}^{k}|\EGP(\mathcal{S}_i)|.\] If
$|\mathcal{P}_\mathrm{i}|> |V_\mathrm{i}|$, then
\begin{eqnarray*}
|U(\mathcal{S})|&=&|\mathcal{P}|\\
&>&|V_\mathrm{i}|+\sum_{i=1}^{k}|EGP(\mathcal{S}_i)|\\
&\geqslant&|V_\mathrm{i}|+\sum_{i=1}^{k}m_i\\
&=&\gamma.
\end{eqnarray*}
This is a contradiction. \qed\end{proof}

\bigskip

\begin{lemma}
\label{least cliques which inland vertices need} If $G\neq K_3$,
then $|\mathcal{P}_\mathrm{i}|\geq |V_\mathrm{i}|- \sum_{v\in
V_{cw}}w_v(\mathcal{P})$.
\end{lemma}
\begin{proof}
Let $V_2'(G)=\{v\in V_2(G): C'_v\subseteq \mathcal{P}_\mathrm{c}\}$,
where $C'_v$ is the set obtained by removing all trivial cliques
from $C_v$. For all $v \in V_2(G) \setminus V_2'(G)$, it follows
that $C_v\cap \mathcal{P}_\mathrm{i}\neq\es$ implies that
$|\mathcal{P}_\mathrm{i}| \geq |V_2(G)| - |V'_2(G)|$.

\medskip

\noindent By the construction of $f_1$, $V_2'(G)\cap
R_\mathcal{P}=\es$. We claim that $v\in NW_\mathcal{P}\cap V_2'(G)$
implies that $v \in V_\mathrm{cnw}$. The reason is the following.

\medskip

\noindent If $v\in NW_\mathcal{P}\cap V_2'(G)$, then all nontrivial
cliques in $C_v$ are also in $\mathcal{P}_\mathrm{c}$. By the
construction of $f_2$, all those cliques are induced by $v$-stars.
Hence $v\in V_\mathrm{c}$. This proves the claim.

\medskip

\noindent By the construction of $f_3$, for each $v\in V_{cw}$,
there are exactly $1+w_v(\mathcal{P})$ vertices in $V_2'(G)$.
Moreover, for each stalk vertex $v\in W\setminus V_\mathrm{c}$, no
vertex in a $v$-wing is in $V_2'(G)$. This proves the correctness of
the following derivation:
\begin{align*}
|\mathcal{P}_\mathrm{i}| & \geq  |V_2(G)| - |V'_2(G)|\\
& =  |V_2(G)| - (|V_{cnw}|+\sum_{v\in V_{cw}}(1+w_v(\mathcal{P})))\\
& =  |V_2(G)| - (|V_{cnw}| +|V_{cw}|+\sum_{v\in V_{cw}}w_v(\mathcal{P}))\\
& =  |V_2(G)| - (|V_\mathrm{c}| +\sum_{v\in V_{cw}}w_v(\mathcal{P}))
\tag{by Proposition~\ref{all members of Q(F_C^) are substars of Q(F)}}\\
& =  |V_\mathrm{i}|-\sum_{v\in V_{cw}}w_v(\mathcal{P}).
\end{align*}
This completes the proof.
\qed\end{proof}

\bigskip

\begin{lemma}
\label{no critical vertex in W} If $G$ is not a $\mathrm{TP}_1$ and
$|U(\mathcal{S})|\leq \gamma$, then $w_v(\mathcal{P})=0$ for every
$v\in V_{cw}$.
\end{lemma}
\begin{proof}
Suppose to the contrary there exists a $v \in V_{cw}$ with
$w_v(\mathcal{P})\neq 0$.
Let $u_1,\dots ,u_m$ be $v$'s
neighbors that have
degree one. For $1 \leq i \leq w_v(\mathcal{P})$,
let $vx_{2i-1}x_{2i}$
be the $v$-wings which induce 3-cliques in
$\mathcal{P}$.

\medskip

\noindent
Let $G'$ be the subgraph of $G^*$ with
\begin{align*}
V(G^{\prime})&=\{vu_1, vu_2, \ldots, vu_m, vx_1, vx_2,\ldots
, vx_{2w_v(\mathcal{P})}\}\\
E(G^{\prime})&=E(G[V^{\prime}]) \setminus
\{(vx_{2i-1},vx_{2i}): 1\leq
i\leq w_v(\mathcal{P})\}.
\end{align*}

\medskip

\noindent
By Corollary~\ref{deBr}, any
edge-clique partition of $G'$ has at least $m+w_v(\mathcal{P})$ cliques
since there exists a clique of size
$m+w_v(\mathcal{P})$ in $G'$. Moreover, any
clique in $\mathcal{P}$ which contains edges of $G'$ is clearly
induced by a $v$-star. Thus $\mathcal{P}$ has at least
$m+w_v(\mathcal{P})$ cliques induced by $v$-stars.

\medskip

\noindent
The case where $\mathcal{P}$ has exactly $m+w_v(\mathcal{P})$ cliques
induced by $v$-stars occurs only when $w_v(\mathcal{P})=1$. The
reason is the following.

\medskip

\noindent
When $w_v(\mathcal{P})=1$, the edges
of $G'\cup (vx_1,vx_2)$ are partitioned by an
$\NP$ consisting of:
\begin{enumerate}[\rm (1)]
\item a
clique of size $m+2w_v(\mathcal{P})-1$ induced by the set of vertices
$V(G')\setminus \{vx_1\}$ and,
\item $m+2w_v(\mathcal{P})-1$ $K_2$'s intersecting in vertex $vx_1$.
\end{enumerate}
Note that the 2-clique $\{vx_1,vx_2\}$ is contained in the $K_3$ in
$\mathcal{P}$ induced by $vx_1x_2$. Thus $\mathcal{P}$ has exactly
$m+w_v(\mathcal{P})$ cliques induced by $v$-stars only when
$w_v(\mathcal{P})=1$.

\medskip

\noindent Consider the  case where $w_v(\mathcal{P})=1$. Since $G$
is not a TP$_1$, there exists another $v$-wing. Thus $v$ has a
neighbor
\[y\notin \{u_1,u_2,\ldots ,u_m,x_1,x_2\}.\]
Add the vertex $vy$ to $G'$ together with all those edges in
$E(G^*)$ that have endpoint $vy$ and the other endpoint in $V(G')$.
Let $G^+$ be this subgraph of $G^*$. Since $y$ is not adjacent to
any vertex in $\{u_1,u_2,\ldots ,u_m,x_1,x_2\}$ of $G$, any clique
in $\mathcal{P}$ which contains edges of $G^+$ is still induced by a
$v$-star. Therefore, by Corollary~\ref{deBr}, $\mathcal{P}$ has at
least $m+2$ cliques induced by $v$-stars. That is, if
$w_v(\mathcal{P}) \neq 0$, then $\mathcal{P}$ contains more than
$m+w_v(\mathcal{P})$ cliques induced by $v$-stars.

\medskip

\noindent On the other hand, by Proposition~\ref{all members of
Q(F_C^) are substars of Q(F)}, for each $v_i \in
V_\mathrm{c}\setminus W \cup V_{cw}$ with $w_{v_i}(\mathcal{P})=0$,
$\mathcal{P}$ contains at least $m_i$ cliques induced by
$v_i$-stars. It follows that
\[|\mathcal{P}_\mathrm{c}| > \sum_{i=1}^k
m_i+\sum_{v\in V_{cw}}w_v(\mathcal{P}).\] Thus by Lemma~\ref{least
cliques which inland vertices need},
\begin{align*}
|U(\mathcal{S})|&= |\mathcal{P}|\\
&=|\mathcal{P}_\mathrm{i}| + |\mathcal{P}_\mathrm{c}|\\
& >  |V_\mathrm{i}|-\sum_{v\in V_{cw}}w_v(\mathcal{P}) + \sum_{i=1}^k
m_i+\sum_{v\in V_{cw}}w_v(\mathcal{P})=\gamma.
\end{align*}
This contradiction concludes our proof.
\qed\end{proof}

\bigskip

\begin{corollary}
\label{f^{-1}(inland) not empty after delete critical star} If $G$
is not a $\mathrm{TP}_1$ and $|U(\mathcal{S})|\leq \gamma$, then
$C_v\cap \mathcal{P}_\mathrm{i}\neq \es$ for all $v \in V_i$.
\end{corollary}
\begin{proof}
Let $v\in V_\mathrm{i}$. First assume that $v \in R_\mathcal{P} \cup
NW_\mathcal{P}$. By the construction of $f_1$ and $f_2$, we can
obtain that $C_v$ has a clique induced by a triangle or a nontrivial
$v$-star in $G$. Such a clique is in $\mathcal{P}_\mathrm{i}$ and
therefore the corollary holds in this case.

\medskip

\noindent Now assume that $v\in V_{cw}$. Since $w_v(\mathcal{P})=0$
by Lemma~\ref{no critical vertex in W}, every non-stalk vertex $x$
of the $v$-wings (which is in $V_\mathrm{i}$) is assigned by $f_3$
to the clique induced by $S_x^2$ (which is in
$\mathcal{P}_\mathrm{i}$). For a stalk vertex $v\in W\setminus
V_\mathrm{c}$, all cliques induced by $v$-wings or $v$-stars are in
$\mathcal{P}_\mathrm{i}$. Therefore, by the construction of $f_3$,
for all vertices $x$ in $v$-wings, $C_x\cap
\mathcal{P}_\mathrm{i}\neq \es$. This completes the proof.
\qed\end{proof}

\bigskip

\begin{lemma}
\label{limit on inland vertex and critical star} If $G$ is not a
$\mathrm{TP}_1$ and $|U(\mathcal{S})|\leq \gamma$, then the
following statements hold.
\begin{enumerate}[\rm (1)]
\item $|\mathcal{P}_\mathrm{i}|= |V_\mathrm{i}|$,
\item for $i \in [k]$ and $v_i \in V_c$,
$\mathcal{P}$ contains exactly $m_i$ cliques
induced by $v_i$-stars,
\item for every $v\in V_\mathrm{i}\setminus V_{3w}$,
$\mathcal{P}$ contains at most one nontrivial clique induced by
$v$-stars unless $G=3K_2\vee K_1$.%
\footnote{See Definition~\vref{df v-stars} to recall the
definition of $V_{3w}$.}
\end{enumerate}
\end{lemma}
\begin{proof}
By Lemmas~\ref{least cliques which inland vertices need} and%
~\ref{no critical vertex in W}, $|\mathcal{P}_\mathrm{i}|\geq
|V_\mathrm{i}|$. By Corollary~\vref{Q-I < V-I} the first statement
holds.

\medskip

\noindent If, for some $i \in [k]$, $\mathcal{P}$ contains more than
$m_i$ cliques induced by $v_i$-stars, then, by the first statement
and Proposition~\vref{all members of Q(F_C^) are substars of Q(F)},
\[|U(\mathcal{S})|=|\mathcal{P}|=|\mathcal{P}_\mathrm{i}|+|\mathcal{P}_\mathrm{c}|>\gamma.\]
This is a contradiction.
Therefore the second statement
holds.

\medskip

\noindent Consider the last statement. Suppose that $G\neq 3K_2\vee
K_1$ and suppose that there are two cliques induced by $v$-stars in
$\mathcal{P}$, for some vertex $v\in V_\mathrm{i}\setminus V_{3w}$.
By the construction of $f$, the set $C_v$ contains these two cliques
unless $v$ is a stalk vertex in $W$.

\medskip

\noindent First assume that $v$ is a stalk vertex in $W$ and that
$w_v(\mathcal{P})> 0$. By Lemma~\vref{wings}, the set $\mathcal{P}$
contains more than $w_v(\mathcal{P})+1$ cliques induced by
$v$-stars.

\medskip

\noindent Now assume that $w_v(\mathcal{P})=0$. The set
$\mathcal{P}$ still contains more than one clique induced by a
$v$-star, since $w_v(\mathcal{P})+1=1$. We conclude that the number
of cliques induced by $v$-wings and $v$-stars is greater than the
number of vertices in the $v$-wings.

\medskip

\noindent Since $v\in V_\mathrm{i}$, all cliques induced by
$v$-wings and $v$-stars are in $\mathcal{P}_\mathrm{i}$.
Consequently, by the construction of $f_3$, $|C_x\cap
\mathcal{P}_\mathrm{i}|\geq 2$ for some vertex $x$ in a $v$-wing.
However, in that case, $|\mathcal{P}_\mathrm{i}|>|V_\mathrm{i}|$ by
Corollary~\vref{f^{-1}(inland) not empty after delete critical
star}. This contradicts the first statement.

\medskip

\noindent Assume that $v$ is not a stalk vertex in $W$. The set
$C_v$ contains two cliques induced by $v$-stars, that is, $|C_v\cap
\mathcal{P}_\mathrm{i}|\geq 2$. This is also a contradiction. This
completes the proof. \qed\end{proof}

\bigskip

\begin{lemma}
\label{no T-triangle} If $G$ is neither $K_4$ nor a $\mathrm{TP}_1$
and $|U(\mathcal{S})|\leq \gamma$, then
$\mathcal{P}_{r}=\es$.\footnote{Recall the definition of
$\mathcal{P}_r$, it's Item~\ref{P_r} on Page~\pageref{P_r}.}
\end{lemma}
\begin{proof}
Suppose that $\mathcal{P}_{r}\neq \es$. First assume that
$R_\mathcal{P}\neq \es$. By Lemma~\vref{Q
>= R}, $|\mathcal{P}_{r}|>|R_\mathcal{P}|$. By the
construction of $f_1$, this implies that $|C_v\cap
\mathcal{P}_{r}|\geq 2$ for some $v\in R_\mathcal{P}$.

\medskip

\noindent Now assume that $R_\mathcal{P}=\es$, namely, there is no
$f_1$. However, there is at least one clique in $\mathcal{P}_{r}$.
Since $\mathcal{P}_{r}\subset\mathcal{P}_\mathrm{i}$, both cases
imply $|\mathcal{P}_\mathrm{i}|> |V_\mathrm{i}|$ by
Corollary~\ref{f^{-1}(inland) not empty after delete critical star}.
This contradicts the first statement in Lemma~\ref{limit on inland
vertex and critical star}. This proves the lemma. \qed\end{proof}

\bigskip

\begin{lemma}
\label{no semiwing} If $G$ is neither a $W_t$ nor a $\mathrm{TP}_1$
and if $|U(\mathcal{S})|\leq \gamma$ then $\mathcal{P}$ contains no
$K_3$ induced by a semiwing in $G$.
\end{lemma}
\begin{proof}
Assume that there is a $K_3$ in $\mathcal{P}$ which is induced by a
semiwing $wuv$, with $d(w)=2$. By Lemma~\vref{semiwing},
$\mathcal{P}$ contains two cliques induced by $u$-stars. Clearly,
$u\notin V_{3w}$ and $G\neq 3K_2\vee K_1$. Thus by the third
statement in Lemma~\ref{limit on inland vertex and critical star},
$u\in V_\mathrm{c}$.

\medskip

\noindent
Let $u_1,\ldots ,u_m$ be $u$'s neighbors that have degree one.
By Corollary~\ref{deBr}, $\mathcal{P}$ has at least $m+1$
$u$-stars partitioning
the edges of
\[G^*[\{uu_1,\dots ,uu_m,uv,uw\}]\]
besides the edge $(uv,uw)$. This
contradicts the second statement of Lemma~\ref{limit on inland vertex and
critical star}.

\medskip

\noindent
This proves the lemma.
\qed\end{proof}

\bigskip

\begin{lemma}
\label{no wing} If $G$ is neither $3K_2\vee K_1$ nor a
$\mathrm{TP}_1$ and if $|U(\mathcal{S})|\leq \gamma$, then
$\mathcal{P}$ contains no $K_3$ that is induced by a $v$-wing for
$v\notin V_{3w}$.
\end{lemma}
\begin{proof}
Suppose that some $v$-wing, for some $v\notin V_\mathrm{3w}$ induces
a $K_3$ in $\mathcal{P}$. By Lemma~\ref{no critical vertex in W},
$v\notin V_\mathrm{c}$. By Lemma~\ref{wings}, $\mathcal{P}$ contains
at least two cliques induced by $v$-stars. This contradicts the
third statement in Lemma~\ref{limit on inland vertex and critical
star}. This proves the lemma. \qed\end{proof}

\bigskip

\begin{corollary}
\label{the only triangles in Q(F) are 3-wing} If $G \notin \{K_4,
W_t, 3K_2\vee K_1, \mathrm{TP}_1\}$ and $|U(\mathcal{S})|\leq
\gamma$, then any clique in $\mathcal{P}$ is induced by either a
star or 3-wing in $G$.
\end{corollary}

\bigskip

\begin{theorem}
\label{LessThanGamma} If $G \notin \{K_4, W_t, 3K_2\vee K_1, \text{a
star}, \mathrm{TP}_1\}$, then
\begin{equation*}
\theta_\mathrm{sd}(G^*)=|V_\mathrm{i}|+\sum_{i=1}^{k}m_i.
\end{equation*}
\end{theorem}
\begin{proof}
Recall that $|V_\mathrm{i}|+\sum_{i=1}^{k}m_i=\gamma$. First
consider a vertex $v\in V_\mathrm{i}$ which is not in a 3-wing. By
the third statement in Lemma~\ref{limit on inland vertex and
critical star}, the set $\mathcal{P}$ contains at most one clique
induced by $v$-star. Thus if $\mathcal{P}$ does not contain the
clique induced by the saturated $v$-star, then it contains a clique
induced by a triangle that contains $v$. By Corollary~\ref{the only
triangles in Q(F) are 3-wing}, $v$ is in a 3-wing, a contradiction.
Thus for each $v\in V_\mathrm{i}$ which is not in a 3-wing,
$\mathcal{P}$ contains the clique induced by the saturated $v$-star.

\medskip

\noindent Now consider a vertex $v_i\in V_\mathrm{c}$ for $i \in
[k]$. Since $G$ is not a star, $v_i$ has a neighbor, say $z$, with
$d(z)\geq 2$. By the second statement in Lemma~\ref{limit on inland
vertex and critical star} and Proposition~\ref{all members of
Q(F_C^) are substars of Q(F)}, $|EGP(\mathcal{S}_i)|=m_i$.
Therefore, by Remark~\ref{remark SD}, $\EGP(\mathcal{S}_i)$ is an
$\NP$, a $\PP$, or an edge-clique cover containing a $K_{m_i}$ and
$m_i-1$ trivial cliques. For the first two cases, since $v_iz$ is
adjacent to all of $v_iv_i^1,\ldots ,v_iv_i^{m_i}$ in $G^*$,
$S(v_iz)$ contains at least two elements, say $a$ and $b$, in
$U(\mathcal{S}_i)$. However, by definition of $\PP$, any two cliques
in an $\NP$ or $\PP$ intersect at exactly one vertex. Suppose the
two cliques in $\EGP(\mathcal{S}_i)$ corresponding to $a$ and $b$ in
the $\EGP$ bijection intersect at a vertex $v_iv_i^{\ell}$ in $G^*$.
This means that $a,b\in S(v_iv_i^{\ell})$ and $|S(v_iz)\cap
S(v_iv_i^{\ell})|\geq 2$, a contradiction. Thus we conclude that,
for each $v_i\in V_\mathrm{c}$, $\mathcal{P}$ contains the clique
induced by the saturated $v_i$-star and $m_i-1$ trivial cliques
induced by $\{v_iv_i^1\},\ldots ,\{v_iv_i^{m_i}\}$.

\medskip

\noindent
We conclude that, if $|U(\mathcal{S})|\leq \gamma$, then
$\mathcal{P}$ consists of $\sum_{i=1}^k (m_i-1)$ trivial cliques and
$|V_2(G)|$ cliques induced by saturated $v$-stars for all $v\in
V_2(G)$ except that, for each 3-wing, say $vxy$, with $d(v)=3$ in
$G$, $\mathcal{P}$ may contain either a $K_3$ induced by $vxy$ and
two $K_2$'s induced by $v$-stars or three cliques induced by
saturated $v$-, $x$- and $y$-stars, respectively.

\medskip

\noindent Accordingly, consider an $\mathcal{S}$ with its
corresponding $\mathcal{P}$ as described above. Every
$S(uv)\in\mathcal{S}$ contains two elements corresponding to two
cliques in $\mathcal{P}$ where both $d(u)$ and $d(v)$ are greater
than or equal to 2 and $uv\in E(G)$ does not connect the two
non-stalk vertices of a wing. Clearly, $S(uv)$ is the unique set in
$\mathcal{S}$ containing these two elements. Due to the elements in
$U(\mathcal{S})$ corresponding to trivial cliques in $\mathcal{P}$,
$S(vu_1)\neq S(vu_2)$ for any pair of vertices $vu_1$ and $vu_2$ in
$G^*$ with $d(u_1)=d(u_2)=1$. Thus $\mathcal{S}\in
F_\mathrm{sd}(G^*)$. Hence
\[\theta_{sd}(G^*)=|V_2(G)|+\sum_{i=1}^k
(m_i-1)=|V_\mathrm{i}|+|V_\mathrm{c}|+\sum_{i=1}^k (m_i-1)=\gamma.\]
This completes the proof.
\qed\end{proof}

\bigskip

\begin{theorem}
\label{line graph m.s.d.-rep.} The type of $F_{sd}(G^*)$ is
\[\tau_{sd}(G^*)=
\begin{cases}
2 & \text{if $G$ is $K_4$, $W_t$, or a $\mathrm{TP}_1$,}\\
3 & \text{if $G$ is $3K_2\vee K_1$,}\\
2+N_\mathrm{PP}(d(v), r) &
\text{if $G$ is a $v$-star with $d(v)\geq 3$,}\\
2^{|V_\mathrm{3w}|} & \mbox{otherwise.}
\end{cases}\]
\end{theorem}
\begin{proof}
It is easy to verify that
\[\tau_{sd}(G^*)=
\begin{cases}
2 &
\text{if $G$ is $K_4$ or $W_t$,}\\
3 &
\text{if
$G=3K_2\vee K_1$.}
\end{cases}\]
By Theorem~\vref{SDomega K_n}, $\tau_{sd}(G^*)=2+N_{\PP}(d(v),r)$
when $G$ is a $v$-star with $d(v)\geq 3$. The `otherwise'-statement
follows directly from the proof of Theorem~\ref{LessThanGamma}.

\medskip

\noindent It remains to prove that $\tau_{sd}(G^*)=2$ when $G$ is a
TP$_1$. Let $G$ be a TP$_1$ that is not $K_3$. There are two classes
of $sd$-set representations for $G$, depending on whether
$\mathcal{P}$ contains a $K_3$ induced by a wing or not.

\medskip

\noindent First, we consider the case where $\mathcal{P}$ contains a
$K_3$ induced by a wing. Let $vxy$ be such a $v$-wing and let $G'$
be the subgraph of $G^*$ induced by the saturated $v$-star but
without the edge $(vx,vy)$. By Corollary~\vref{deBr}, $\mathcal{P}$
contains at least $m+1$ cliques in $G'$ which are induced by
$v$-stars, where
\[m=|\{u: u \in V(G) \text{ and } d(u)=1\}|.\]
The case where $\mathcal{P}$ contains exactly $m+1$ cliques induced
by $v$-stars occurs only when the $m+1$ cliques together with
$\{vx,vy\}$, form an $\NP$. Thus, in this case, $\tau_{sd}(G^*)=2$.

\medskip

\noindent Consider the case where $\mathcal{P}$ does not contain a
$K_3$ induced by a wing. By using an argument similar to the one in
Theorem~\ref{LessThanGamma}, we also easily derive
$\tau_{sd}(G^*)=2$. This completes the proof. \qed\end{proof}

\section{The types $\tau_{sa}(G^*)$ and
$\tau_{sdu}(G^*)$}
\label{thetaF_sa,su,sdu}

In this section, we assume that $\mathcal{S}\in F_{sa}(G^*)$. The
set $S(uv)$ denotes the set in the set representation which is
assigned to $uv\in V(G^*)$ with respect to $\mathcal{S}$. Recall
that $\mathcal{S}_i=\{S(v_iv_i^1), S(v_iv_i^2), \ldots,
S(v_iv_i^{m_i})\}$, where $v_i\in V_\mathrm{c}$. Let
$\mathcal{P}=\EGP(\mathcal{S})$ and let
$\gamma'=|V_\mathrm{i}|+\sum_{i=1}^{k}(m_i+1)$.

\bigskip

\begin{lemma}
\label{Q-IA < V-IA} If $G$ is not a star and $|U(\mathcal{S})|\leq
\gamma'$, then $|\mathcal{P}_\mathrm{i}|\leq |V_\mathrm{i}|$.
\end{lemma}
\begin{proof}
Assume $|\EGP(\mathcal{S}_i)|=m_i\geq 3$ for some $i \in [k]$. By
Theorem~\ref{SAomega K_n}, $\EGP(\mathcal{S}_i)$ is an $\NP$ or
$\PP$. An argument similar as in Theorem~\ref{LessThanGamma} gives a
contradiction. Thus we only need to consider $m_i\in \{1,2\}$.

\medskip

\noindent If $m_i=1$ for some $i \in [k]$, then
$|\EGP(\mathcal{S}_i)|\geq 2$; otherwise we would have that, for any
$z$ adjacent to $v_i$ with $d(z) \geq 2$, $S(v_iv_i^1)\subset
S(v_iz)$, which is a contradiction.

\medskip

\noindent Assume $m_i=2$ for some $i \in [k]$. We can obtain that
$|\EGP(\mathcal{S}_i)|\geq 3$ since $\theta_{sa}(K_2)=3$. As a
consequence, $|\EGP(\mathcal{S}_i)|\geq m_i+1$ for $i \in [k]$. By
Proposition~\vref{all members of Q(F_C^) are substars of Q(F)} and
the fact that
$|U(\mathcal{S})|=|\mathcal{P}_\mathrm{i}|+|\mathcal{P}_\mathrm{c}|$,
we obtain $|\mathcal{P}_\mathrm{i}|>|V_\mathrm{i}|$. This implies
that $|U(\mathcal{S})|> \gamma'$ which is a contradiction. This
completes the proof. \qed\end{proof}

\bigskip

Assume $G \neq K_3$.
In the following, we construct another surjection
\[f':\mathcal{P}\rightarrow V_2(G).\]
The construction goes by the same rules as in Section 4, except that
the construction of $f_3$ is slightly modified.

\bigskip

If $\mathcal{P}$ contains a $K_3$ induced by a wing $vxy$, say with
stalk vertex $v$ in $G$, then $S(xy)$ must contain a
monopolist.\footnote{A monopolist was defined in Definition~\vref{df
monopolist}.} In this case, in constructing $f_3$, assign the
trivial clique in $\mathcal{P}$ which corresponds to the monopolist
to one of $x$ and $y$, and the $K_3$ in $\mathcal{P}$ induced by
$vxy$ to the other of $x$ and $y$. Let $f'_3$ be the modified $f_3$.
Based on the adjustment, we have the following lemma.

\begin{lemma}
\label{f^{-1}(inland) not empty after delete critical star--s.a.} If
$G\neq K_3$, then, for all $v \in V_\mathrm{i}$, $C_v \cap
\mathcal{P}_\mathrm{i}\neq \es$ and $|\mathcal{P}_\mathrm{i}|\geq
|V_\mathrm{i}|$.
\end{lemma}
\begin{proof}
For every $v\in V_\mathrm{i} \cap R_{\mathcal{P}}$, $C_v$ is
assigned a clique from $\mathcal{P}_\mathrm{i}$ by $f_1$, since
$f_1$ is surjective. Similarly, for every $v \in V_\mathrm{i} \cap
NW_{\mathcal{P}}$, $C_v$ is assigned a clique from
$\mathcal{P}_\mathrm{i}$ by $f_2$ and for every $v \in V_i \cap W$,
$C_v$ is assigned a clique from $\mathcal{P}_i$ by $f'_3$.

\medskip

\noindent
That proves the lemma.
\qed\end{proof}

\begin{lemma}
\label{limit on inland vertex and critical star--s.a.}
If $G$ is neither a $K_3$ nor a star and
if $|U(\mathcal{S})|\leq \gamma'$, then the following statements hold.
\begin{enumerate}[\rm (1)]
\item $|\mathcal{P}_\mathrm{i}|=|V_\mathrm{i}|$.
\item For $i \in [k]$, $|\EGP(\mathcal{S}_i)|=m_i+1$
and
$\mathcal{P}$ contains precisely $m_i+1$ cliques induced by $v_i$-stars.
\item For every $v\in V_\mathrm{i}$, $\mathcal{P}$ contains at most
one clique induced by $v$-star.
\end{enumerate}
\end{lemma}
\begin{proof}
Lemma~\ref{Q-IA < V-IA} implies that $|\EGP(\mathcal{S}_i)|\geq
m_i+1$ for $i \in [k]$. By Proposition~\ref{all members of Q(F_C^)
are substars of Q(F)}, $\mathcal{P}$ contains at least $m_i+1$
cliques induced by $v_i$-stars for $i \in [k]$. If either
$|\mathcal{P}_\mathrm{i}| \neq |V_\mathrm{i}|$ or $\mathcal{P}$
contains more than $m_i+1$ cliques induced by $v_i$-stars for some
$i \in [k]$, then, by Lemma~\ref{f^{-1}(inland) not empty after
delete critical star--s.a.},
\[|U(\mathcal{S})|=|\mathcal{P}|=|\mathcal{P}_\mathrm{i}|+
|\mathcal{P}_\mathrm{c}|>|V_\mathrm{i}|+\sum_{i=1}^{k}(m_i+1)=\gamma'.\]
This contradicts the assumption that $|U(\mathcal{S})|\leq \gamma'$.
Thus $|\mathcal{P}_\mathrm{i}|=|V_\mathrm{i}|$ and $\mathcal{P}$
contains precisely $m_i+1$ cliques induced by $v_i$-stars for $i \in
[k]$. Consequently, for all $i \in [k]$,
$|\EGP(\mathcal{S}_i)|=m_i+1$. This proves the the first two
statements.

\medskip

\noindent It remains to prove the third statement. By
Lemma~\ref{f^{-1}(inland) not empty after delete critical
star--s.a.}, for all $v \in V_\mathrm{i}$, it follows that $C_v\cap
\mathcal{P}_\mathrm{i}\neq \es$ and $|\mathcal{P}_\mathrm{i}|
\geqslant |V_\mathrm{i}|$. Suppose that $\mathcal{P}$ contains two
cliques induced by $v$-stars for some $v\in V_\mathrm{i}$. By the
construction of $f'$, it follows that $C_v$ contains those two
cliques unless $v$ is a stalk vertex in $W$. However, if $v$ is a
stalk vertex in $W$, then $C_{v'}$, for some $v'$ on the $v$-wings,
contains two cliques in $\mathcal{P}_\mathrm{i}$ by the construction
of $f'_3$. Thus, in any case, there exists a $v\in V_\mathrm{i}$
with $|C_{v}\cap \mathcal{P}_\mathrm{i}|\geq 2$. This results in
$|\mathcal{P}_\mathrm{i}|
> |V_\mathrm{i}|$ which
contradicts the first statement. This completes the proof.
\qed\end{proof}

\bigskip

\begin{lemma}
\label{no T-triangle--s.a.} If $G \notin \{K_3, K_4, \text{ a
star}\}$ and $|U(\mathcal{S})|\leqslant \gamma'$, then
$\mathcal{P}_{r}=\es$.
\end{lemma}
\begin{proof}
This lemma follows by replacing Corollary~\ref{f^{-1}(inland) not
empty after delete critical star} and Lemma~\ref{limit on inland
vertex and critical star} by Lemmas~\ref{f^{-1}(inland) not empty
after delete critical star--s.a.} and~\ref{limit on inland vertex
and critical star--s.a.}, respectively, in the proof of
Lemma~\ref{no T-triangle}. \qed\end{proof}

\bigskip

\begin{lemma}
\label{no semiwing--s.a.} Assume that $G \notin \{K_3, W_t, \text{ a
star}\}$ and assume that $|U(\mathcal{S})|\leq \gamma'$. If a
semiwing in $G$ induces a $K_3$ in $\mathcal{P}$, then at least one
stalk vertex of this semiwing is in $V_\mathrm{c}$.
\end{lemma}
\begin{proof}
Let $wuv$ with $d(w)=2$ be a semiwing which induces a $K_3$ in
$\mathcal{P}$. By Lemma~\ref{semiwing}, at least two $u$-stars
induce nontrivial cliques in $\mathcal{P}$. By the third statement
in Lemma~\ref{limit on inland vertex and critical star--s.a.}, $u\in
V_\mathrm{c}$. This completes the proof. \qed\end{proof}

\bigskip

\begin{lemma} \label{no wing--s.a.}
Assume that $G \notin \{K_3, \text{ a star}\}$ and assume that
$|U(\mathcal{S})|\leq \gamma'$. If there is a $v$-wing in $G$
inducing a $K_3$ in $\mathcal{P}$, then $v \in V_\mathrm{c}$.
\end{lemma}
\begin{proof}
By Lemma~\ref{wings}, $\mathcal{P}$ contains two cliques induced by
$v$-stars. By the third statement in Lemma~\ref{limit on inland
vertex and critical star--s.a.}, $v\in V_\mathrm{c}$. This completes
the proof. \qed\end{proof}

By Corollary~\ref{deBr} and Theorem~\ref{bridges}, if there exists
an $\EGP(K_{m_i})$ such that $1<|\EGP|\leq m_i+1$ and there is no
trivial clique in it, then it can only be an $\NP$, $\PP$,
$1$-punctured $\PP$, or $2$-punctured Fano Plane. Thus, by the
second statement in Lemma~\ref{limit on inland vertex and critical
star--s.a.}, for each $i \in [k]$, $\EGP(\mathcal{S}_i)$ can only be
one of the following five edge-clique covers: $\NP$ plus one trivial
clique, $\PP$ plus one trivial clique, $1$-punctured PP,
$2$-punctured Fano Plane, or $K_{m_i}$ plus $m_i$ distinct trivial
cliques. However, both of $\PP$ plus one trivial clique and
$2$-punctured Fano Plane are impossible as we will show later. We
distinguish the following types of edge-clique covers for
$\EGP(\mathcal{S}_i)$.

\begin{description}
\item[Type I:] $\EGP(\mathcal{S}_i)$ is an $\NP$ plus one trivial
clique.

In this case, the unique way to configure the set representation
is as follows. The clique $K_{m_i-1}$ in $\NP$ and the trivial
clique $K_1$ are disjoint in $\EGP(\mathcal{S}_i)$. Furthermore,
every $S(v_iv_i^j)$ for $m_i+1\leqslant j\leqslant d(v_i)$
contains the two elements in $U(\mathcal{S}_i)$ corresponding to
$K_{m_i-1}$ and $K_1$. Notice that $d(v_i) > m_i+1$ implies
$|S(v_iv_i^{m_i+1})\cap S(v_iv_i^{m_i+2})|\geq 2$. Thus this
case occurs only when $d(v_i)=m_i+1$.
\item[Type II:] $\EGP(\mathcal{S}_i)$ is a $1$-punctured $\PP$.

Let $r$ be the order of the $\PP$ and let $x$ be the deleted
point. Thus $m_i=r^2+r$. By Theorem~\ref{project plane}, the
$\PP$ has $r+1$ lines that go through $x$. Therefore,
$\EGP(\mathcal{S}_i)$ has exactly $r+1$ pairwise disjoint
cliques of cardinality $r$. The total number of these vertices
is $r^2+r=m_i$. Of course, every other pair of lines in
$\EGP(\mathcal{S}_i)$ intersect in one point. Thus
$S(v_iv_i^{m_i+1})$ contains the $r+1$ elements in
$U(\mathcal{S}_i)$ corresponding to these $r+1$ cliques. This
case occurs when $d(v_i)=m_i+1$; since otherwise
$|S(v_iv_i^{m_i+1})\cap S(v_iv_i^{m_i+2})|\geq 2$ which is a
contradiction.

\item[Type III:] $\EGP(\mathcal{S}_i)$ consists of $K_{m_i}$ plus $m_i$
trivial cliques.

In this case, $S(v_iv_i^j)$ for each $j$ with $m_i+1\leq
j\leq d(v_i)$ contains either
\begin{enumerate}[\rm (a)]
\item the element in
$U(\mathcal{S}_i)$ that corresponds to the $K_{m_i}$ or,
\item the
elements in $U(\mathcal{S}_i)$ that correspond to the $m_i$
trivial cliques.
\end{enumerate}

\item[Type IV:] $\EGP(\mathcal{S}_i)$ is a $\PP$ plus one trivial clique.

Since a $\PP$ of $K_{m_i}$ does not have a clique $K_{m_i-1}$,
the configuration for Type I cover cannot be applied in this
case. Thus this case is impossible.

\item[Type V:] $\EGP(\mathcal{S}_i)$ is a $2$-punctured Fano Plane.

In a $2$-punctured Fano Plane (see the $\EGP$-set in
Figure~\ref{FLS with 5 points 6 lines} as an example),
$\{Q_3,Q_6\}$ and $\{Q_4,Q_5\}$ are the only sets of pairwise
disjoint cliques. However, if $S(v_iv_i^{m_i+1})$ contains the
two elements in $U(\mathcal{S}_i)$ corresponding to $Q_3$ and
$Q_6$, then it contains another element corresponding to $Q_1$
or $Q_2$. Consequently, $|S(v_iv_i^{m_i+1})\cap S(v_iv_i^j)|\geq
2$ for some $1\leq j\leq m_i$, a contradiction. A similar
reasoning applied to $Q_4$ and $Q_5$ also leads to a
contradiction. Thus this case is also impossible.
\end{description}

\bigskip

\begin{lemma}
\label{d(v_i)=m_i+1}
If $G$ is neither $K_3$ nor a star and $|U(\mathcal{S})|\leq
\gamma'$, then the following two statements are true.
\begin{enumerate}[\rm (1)]
\item
For $i\in [k]$ with $d(v_i)=m_i+1$, the cliques in
$\mathcal{P}$, induced by $v_i$-stars, are either a $K_{d(v_i)}$
with $m_i$ trivial cliques $\{v_iv_i^1\},\dots
,\{v_iv_i^{m_i}\}$ or they constitute an $\NP$ or a $\PP$ of
$K_{d(v_i)}$.
\item For $i \in [k]$ with $d(v_i)>m_i+1$,
the set $\EGP(\mathcal{S}_i)$ consists of a $K_{m_i}$ and $m_i$
distinct trivial cliques. Furthermore, each $S(v_iv_i^j)\cap
U(\mathcal{S}_i)$, for $m_i+1\leq j\leq d(v_i)$, contains either
the element corresponding to the $K_{m_i}$ or the elements
corresponding to the $m_i$ trivial cliques.
\end{enumerate}
\end{lemma}
\begin{proof}
By inspection, Types~I and~II covers occur only when $d(v_i)=m_i+1$.
Type~III covers might have $d(v_i)>m_i+1$. Thus we have to consider
these three types of covers when proving the first statement and we
only need to consider Type~III when proving the second statement.

\medskip

\noindent
By the definition of a Type~III cover, the set $\EGP(\mathcal{S}_i)$ is
a $K_{m_i}$ plus $m_i$ distinct trivial cliques. This implies that
$S(v_iv_i^j)\cap U(\mathcal{S}_i)$ for $m_i+1\leq j\leq
d(v_i)$ contains either the element corresponding to the $K_{m_i}$
or the elements corresponding to the $m_i$ trivial cliques.
This proves
the second statement.

\medskip

\noindent
Next, we prove the first statement.
Consider the case that
$\EGP(\mathcal{S}_i)$ is a Type~I cover. The trivial clique in
$\EGP(\mathcal{S}_i)$ contains the intersection of all the $K_2$'s in the
$\NP$. Furthermore, $S(v_iv_i^{m_i+1})\cap U(\mathcal{S}_i)$
contains the two elements corresponding to the $K_{m_i-1}$ and $K_1$
in $\EGP(\mathcal{S}_i)$. Therefore, $v_iv_i^{m_i+1}$ together with
the vertices in the $K_{m_i-1}$ induce a clique in $\mathcal{P}$.
Vertex $v_iv_i^{m_i+1}$ together with the vertex in $K_1$ induce
another clique in $\mathcal{P}$. As a consequence, the cliques in
$\mathcal{P}$ induced by $v_i$-stars constitute an
$\NP$ of $K_{d(v_i)}$.

\medskip

\noindent
Now consider the case that $\EGP(\mathcal{S}_i)$ is a Type~II
cover. That is, $\EGP(\mathcal{S}_i)$ is a $\PP$ with a vertex, say $x$,
deleted. Furthermore, $S(v_iv_i^{m_i+1})\cap U(\mathcal{S}_i)$
contains the cliques in
$\EGP(\mathcal{S}_i)$ that, originally, contained the vertex $x$.
Therefore, $v_iv_i^{m_i+1}$ together with the vertices in each of
those cliques induce a clique in $\mathcal{P}$. Accordingly, the
cliques in $\mathcal{P}$ induced by $v_i$-stars constitute a $\PP$
of
$K_{d(v_i)}$.

\medskip

\noindent
Finally, we consider the case that $\EGP(\mathcal{S}_i)$ is a Type~III
cover.
In this case, if $d(v_i)=m_i+1$ with $m_i\geq 2$ and
$S(v_iv_i^{m_i+1})$ contains the elements corresponding to the $m_i$
trivial cliques, then clique $\{v_iv_i^{m_i+1},v_iv_i^j\}\in
\mathcal{P}$ for $1\leq j\leq m_i$. This means that the
cliques in $\mathcal{P}$ induced by $v_i$-stars constitute an
$\NP$ of $K_{d(v_i)}$.

\medskip

\noindent If $d(v_i)=m_i+1$ and $S(v_iv_i^{m_i+1})$ contains the
element corresponding to the $K_{m_i}$, then $S_{v_i}^{d(v_i)}$
induces a clique in $\mathcal{P}$. In this way, the cliques in
$\mathcal{P}$ induced by $v_i$-stars are a $K_{d(v_i)}$ and $m_i$
trivial cliques. This completes the proof. \qed\end{proof}

\bigskip

\begin{theorem}
\label{specialgraphs m.s.a.-rep.} For tailed peacocks $G$,
\begin{equation*}
\tau_{sa}(G^*)=
\begin{cases}
5 & \text{if $G$ is $\mathrm{TP}_2$ with $m_1,m_2\geq 2$ and $m_1\neq m_2$,}\\
4 & \text{if $G$ is $\mathrm{TP}_2$ with $m_1=m_2 \geq 2$,}\\
3 & \text{if $G$ is $\mathrm{TP}_2$ with $m_1\geq 2$ and $m_2=1$,}\\
2 & \text{otherwise.}
\end{cases}
\end{equation*}
\end{theorem}
\begin{proof}
Let $G$ be a TP. We count $\tau_{sa}(G^*)$ by analyzing all possible
set representations $\mathcal{S}$ of $G^*$ with
$|U(\mathcal{S})|\leq \gamma'$. We consider two cases.

\medskip

\noindent {\bf Case 1.} $G$ is a TP$_1$ or a TP$_2$.

\medskip

\noindent If $G$ is a TP$_1$ or a TP$_2$, then $G$ has a stalk
vertex $v_1$ or two stalk vertices $v_1$ and $v_2$. By
Lemma~\ref{d(v_i)=m_i+1}, for each stalk vertex $v_i$ with $i \in
[k]$, the set $\EGP(\mathcal{S}_i)$ consists of a $K_{m_i}$ and
$m_i$ distinct trivial cliques, and $S(v_iv_i^{m_i+1})$ and
$S(v_iv_i^{m_i+2})$ contain either the $K_{m_i}$ or the $m_i$
trivial cliques. Note that, in this case, $k\in \{1,2\}$.

\medskip

\noindent When $m_i\geq 2$, the sets $S(v_iv_i^{m_i+1})$ and
$S(v_iv_i^{m_i+2})$ cannot contain the $m_i$ trivial cliques at the
same time. If $k=1$, then $\tau_\mathrm{sa}(G^*)=2$, that is, one
for $S(v_1v_1^{m_1+1})=S(v_1v_1^{m_1+2})$ and the other for
$S(v_1v_1^{m_1+1})\neq S(v_1v_1^{m_1+2})$.

\medskip

\noindent
For the case where $k=2$,
since the triangle in $G$ induces a $K_3$ in $\mathcal{P}$, this
results in $S(v_1v_1^{m_1+1})\ne S(v_1v_1^{m_1+2})$ which further
implies $S(v_2v_2^{m_2+1})\ne S(v_2v_2^{m_2+2})$ and vice versa.
Thus if $k=2$ and $m_1,m_2\geq 2$, then
$\tau_\mathrm{sa}(G^*)=5$, unless $m_1=m_2$. If $k=2$ and
$m_1=m_2\geq 2$, then there are two isomorphic
$\mathcal{S}$'s. Here, for $i \in \{1,2\}$,
\[S(v_iv_i^{m_i+1})\neq S(v_iv_i^{m_i+2})\]
and $S(v_1v_2)\cap U(\mathcal{S}_i)$ contains the
$K_{m_i}$ for one of $i \in \{1,2\}$
and the $m_i$ trivial cliques for the
other $i\in \{1,2\}$.
In this case, $\tau_\mathrm{sa}(G^*)=4$.

\medskip

\noindent If $m_2=1$ and $m_1\geq 2$ (or $m_1=1$ and $m_2\geqslant
2$), then there are two possibilities for $S(v_1v_1^{m_1+j})\cap
U(\mathcal{S}_1)$ for $j\in\{1,2\}$. We obtain
$\tau_\mathrm{sa}(G^*)=3$. If $m_1=m_2=1$, then it is easy to see
that $\tau_\mathrm{sa}(G^*)=2$.

\medskip

\noindent {\bf Case 2.} $G$ is a TP$_\mathrm{d1}$ or a
TP$_\mathrm{d2}$.

\medskip

\noindent In this case, for each $v_i\in V_\mathrm{c}$, all the
vertices $v_iv_i^j$, with $m_i+1\leq j\leq d(v_i)$ and with
$d(v_i^j)=2$, are nonadjacent. Therefore, all $S(v_iv_i^j)\cap
U(\mathcal{S}_i)$ for $m_i+1\leq j\leq d(v_i)$ with $d(v_i^j)=2$
contain the same element. If the set $S(v_iv_i^j)$, with
$d(v_i^j)\geq 3$, contains the same element as above, this yields
$\tau_{sa}(G^*)=2$. If the set $S(v_iv_i^j)$ with $d(v_i^j)\geq 3$
does not contain the same element, then all triangles in $G$ induce
$K_3$'s in $\mathcal{P}$ and, therefore, $S(v_i^jx)$, for all $x$
adjacent to $v_i^j$ with $d(x)=2$, contains a common element which
is not in $S(v_iv_i^j)$. This completes the proof. \qed\end{proof}

\bigskip

\begin{lemma}
\label{LemmaforLessThanGamma'} If $G \notin \{K_3, K_4, W_t, \text{
a star}, \mathrm{TP}\}$ and $|U(\mathcal{S})|\leqslant \gamma'$,
then, for each $v_i\in V_\mathrm{c}$ with $d(v_i)>m_i+1$,
\[|\bigcap_{m_i+1 \leq j \leq d(v_i)} \;
S(v_iv_i^j)\cap U(\mathcal{S}_i)|=1,\] and the common element is
the element corresponding to the clique $K_{m_i}$ in
$\EGP(\mathcal{S}_i)$ when $m_i\ge 2$.
\end{lemma}
\begin{proof}
First consider the case where $m_i\geq 2$. Assume without loss of
generality that $v_iv_i^{m_i+1}$ is a vertex in $G^*$ such that
$S(v_iv_i^{m_i+1})$ contains the $m_i$ trivial cliques in
$\EGP(\mathcal{S}_i)$, where $v_i\in V_\mathrm{c}$ with
$d(v_i)>m_i+1$ (see the second statement in
Lemma~\ref{d(v_i)=m_i+1}).

\medskip

\noindent Thus, for $m_i+2\leq j\leq d(v_i)$, every $S(v_iv_i^j)$
contains the clique $K_{m_i}$ in $\EGP(\mathcal{S}_i)$, for
otherwise $|S(v_iv_i^j)\cap S(v_iv_i^{m_i+1})|>1$. By the second
statement in Lemma~\ref{limit on inland vertex and critical
star--s.a.}, for $m_i+2\leq j\leq d(v_i)$, the element in
$S(v_iv_i^j)\cap S(v_iv_i^{m_i+1})$ corresponds to a $K_3$ in
$\mathcal{P}$ induced by the triangle $v_iv_i^jv_i^{m_i+1}$ and
therefore $v_i^j$ is adjacent to $v_i^{m_i+1}$. Thus, by
Lemma~\ref{no T-triangle--s.a.}, all $v_i^j$, for $m_i+2\leq j\leq
d(v_i)$, are of degree two if $d(v_i)>m_i+2$. For the case where
$d(v_i)=m_i+2$, either $v_i^{m_i+1}$ or $v_i^{m_i+2}$ is of degree
2. Thus we consider the following three cases.

\medskip

\noindent {\bf Case 1.} $d(v_i)>m_i+2$.

\medskip

\noindent In this case, $v_i^{m_i+1}$ must have some neighbors other
than the aforementioned neighbors for otherwise $G$ is a
TP$_\mathrm{d1}$, which is forbidden. We prove that all neighbors of
$v_i^{m_i+1}$ are plumes,\footnote{A plume was defined in
Definition~\ref{df peacocks}.} except $v_i$ and $v_i^x$, for
$m_i+2\leq x\leq d(v_i)$. Suppose that $v_i^{m_i+1}$ has a neighbor
$y$ with
\begin{enumerate}[\rm i.]
\item $d(y)\geq 2$,
\item $y\neq v_i$, and
\item $y$ is not
adjacent to $v_i$.
\end{enumerate}
Thus the edges $(v_i^{m_i+1}y, v_i^{m_i+1}v_i)$ and $(v_i^{m_i+1}y,
v_i^{m_i+1}v_i^{m_i+2})$ in $G^*$ are covered by cliques in
$\mathcal{P}$ that are induced by distinct $v_i^{m_i+1}$-stars,
since $y$ is not adjacent to $v_i^{m_i+2}$ and
$v_iv_i^{m_i+1}v_i^{m_i+2}$ induces a triangle in $\mathcal{P}$.
Therefore, $v_i^{m_i+1}\in V_\mathrm{c}$ by the third statement in
Lemma~\ref{limit on inland vertex and critical star--s.a.}.

\medskip

\noindent For brevity, let $v_z$ denote the critical vertex
$v_i^{m_i+1}$ for some $z\in [k]$. By Lemma~\ref{d(v_i)=m_i+1}, the
set $\EGP(\mathcal{S}_z)$ consists of a $K_{m_z}$ and $m_z$ trivial
cliques. Due to the $K_3$ in $\mathcal{P}$ induced by $v_iv_zv_i^x$,
we can obtain that $S(v_zv_i^x)\cap U(\mathcal{S}_z) \neq
S(v_zv_i)\cap U(\mathcal{S}_z)$. That is, either $S(v_zy)\cap
U(\mathcal{S}_z)\neq S(v_zv_i^x)\cap U(\mathcal{S}_z)$ or
$S(v_zy)\cap U(\mathcal{S}_z)\neq S(v_zv_i)\cap U(\mathcal{S}_z)$.
By the second statement in Lemma~\ref{limit on inland vertex and
critical star--s.a.}, either $y$ is adjacent to $v_i^x$ or $y$ is
adjacent to $v_i$. The first option contradicts $d(v_i^x)=2$ and the
second contradicts that $y$ is not adjacent to $v_i$. The conclusion
is that all neighbors of $v_i^{m_i+1}$ are plumes, except $v_i$ and
$v_i^x$, for $m_i+2\leq x\leq d(v_i)$. Thus $G$ is a
TP$_\mathrm{d2}$, a contradiction.

\medskip

\noindent
 {\bf Case 2.} $d(v_i)=m_i+2$ and $d(v_i^{m_i+2})=2$.

\medskip

\noindent A similar argument as in the previous case leads to the
conclusion that $G$ is a TP$_2$; again a contradiction.

\medskip

\noindent
 {\bf Case 3.} $d(v_i)=m_i+2$ and $d(v_i^{m_i+1})=2$.

\medskip

\noindent We prove that $v_i^{m_i+2}$ has no neighbor with degree at
least 2, except $v_i$ and $v_i^{m_i+1}$. Suppose that there exists a
vertex $y$ with
\begin{enumerate}[\rm i.]
\item $d(y)\geq 2$,
\item $y$
is not adjacent to $v_i^{m_i+2}$, and
\item $y\notin \{v_i, v_i^{m_i+1}\}$.
\end{enumerate}
A similar argument as in the first case leads to the conclusion that
either $y$ is adjacent to $v_i^{m_i+1}$ or $y$ is adjacent to $v_i$.
Both options contradict the assumption that titles this case. This
leads again to a TP$_2$, which is still a contradiction.

\medskip

\noindent Finally, consider the case where $m_i=1$. Suppose that
$S(v_iv_i^{m_i+1})$ and $S(v_iv_i^{m_i+2})$ contain the two elements
in $U(\mathcal{S}_i)$, where $v_i\in V_\mathrm{c}$ with
$d(v_i)>m_i+1$. By a similar argument as above, we obtain that
\begin{enumerate}[\rm i.]
\item $v_i^{m_i+1}$ is adjacent to $v_i^{m_i+2}$,
\item the triangle
$v_iv_i^{m_i+1}v_i^{m_i+2}$ induces a $K_3$ in $\mathcal{P}$, and
\item
either $v_i^{m_i+1}$ or $v_i^{m_i+2}$ has degree two.
\end{enumerate}
Without loss of generality, assume that $d(v_i^{m_i+1})=2$. By using
a similar argument as above, we obtain that every neighbor $y \neq
v_i$ of $v_i^{m_i+2}$, with $d(y)\ge 2$ is adjacent to $v_i$. If
there exists such a vertex $y$, which is not $v_i^{m_i+1}$, then
either $S(v_iy)\cap U(\mathcal{S}_i)\ne S(v_iv_i^{m_i+1})\cap
U(\mathcal{S}_i)$ or $S(v_iy)\cap U(\mathcal{S}_i)\ne
S(v_iv_i^{m_i+2})\cap U(\mathcal{S}_i)$ since
$v_iv_i^{m_i+1}v_i^{m_i+2}$ induces a $K_3$ in $\mathcal{P}$.

\medskip

\noindent In the first case, $y$ is adjacent to $v_i^{m_i+1}$, which
contradicts $d(v_i^{m_i+1})=2$. In the second case,
$v_iyv_i^{m_i+2}$ induces a triangle in $\mathcal{P}$. Thus
$d(y)=2$. Moreover, by a similar argument we obtain that all
neighbors $x$ of $v_i$,  with $d(x)\ge 2$ and $x\neq v_i^{m_i+2}$,
are adjacent to $v_i^{m_i+2}$. This yields $d(x)=2$. Thus $G$ is a
TP, a contradiction. This completes the proof. \qed\end{proof}

\bigskip

\begin{theorem}
\label{LessThanGamma'} If $G \notin \{K_3, K_4, W_t, \text{ a star},
\mathrm{TP}\}$, then
$\theta_{sa}(G^*)=|V_\mathrm{i}|+\sum_{i=1}^{k}(m_i+1)$.
\end{theorem}
\begin{proof}
First, we show that there exists an $\mathcal{S}\in F_{sa}(G^*)$
with $|U(\mathcal{S})|\leq \gamma'$. By
Lemma~\ref{LemmaforLessThanGamma'}, the set $\mathcal{P}$ contains
the clique induced by the saturated $v_i$-star for each $i \in [k]$
with $d(v_i)>m_i+1$. Therefore, $\mathcal{P}$ contains no clique
induced by a $K_3$ in $G$ that contains some $v_i\in V_\mathrm{c}$.

\medskip

\noindent
By Lemmas~\ref{no T-triangle--s.a.}-\ref{no
wing--s.a.}, $\mathcal{P}$ contains no clique induced by a $K_3$ in
$G$ and therefore, by the third statement
in Lemma~\ref{limit on inland
vertex and critical star--s.a.}, $\mathcal{P}$ contains the clique
induced by the saturated $v$-star for every $v\in V_\mathrm{i}$.
By
Lemma~\ref{d(v_i)=m_i+1}, there exists an $\mathcal{S}$ with
$|U(\mathcal{S})|\leq \gamma'$.

\medskip

\noindent On the other hand, by using a similar argument as in
Lemma~\ref{LessThanGamma}, an $\mathcal{S}$ with its corresponding
$\mathcal{P}$ as described above is in $F_{sa}(G^*)$. Hence
\[\theta_{sa}(G^*)=\gamma'=|V_\mathrm{i}|+\sum_{i=1}^{k}(m_i+1).\]
This completes the proof.
\qed\end{proof}

\bigskip

\begin{theorem}
\label{line graph m.s.a.-rep.} For a graph $G$ which is not a
$\mathrm{TP}$,
\begin{equation*}
\tau_{sa}(G^*)=
\begin{cases}
1 & \text{if $G$ is $K_3$,}\\
2 & \text{if $G$ is $K_4$ or $W_t$,}\\
1+N_{\PP}(d(v),r) & \text{if $G$ is a $v$-star with $d(v)\geq 3$,}\\
2^x y^z & \mbox{otherwise,}
\end{cases}
\end{equation*}
where $x =|\{v_i: m_i=2 \mbox{ and } d(v_i)=3\}|$,
$y=3+N_{\PP}(m_i+1,r)$ and $z = |\{v_i: m_i\geq 3 \mbox{ and }
d(v_i)=m_i+1\}|$.
\end{theorem}
\begin{proof}
Clearly, $\tau_{sa}(G^*)=3$ when $G=K_3$ and
$\tau_{sa}(G^*)=\tau_{sd}(G^*)=2$ when $G \in \{K_4,W_t\}$. By
Theorem~\ref{SAomega K_n}, $\tau_{sa}(G^*)=1+N_{\PP}(d(v),r)$ when
$G$ is a $v$-star with $d(v)\geq 3$. It remains to prove that
\[\tau_{sa}(G^*)=2^x
y^z\] for the remaining cases, where $x =|\{v_i: m_i=2 \mbox{ and } d(v_i)=3\}|$,
$y=3+N_{\PP}(m_i+1,r)$ and $z = |\{v_i: m_i\geq 3 \mbox{ and }
d(v_i)=m_i+1\}|$. Recall that for each $i \in [k]$
with $d(v_i)=m_i+1$, the cliques in $\mathcal{P}$ induced by
$v_i$-stars are either
\begin{enumerate}[\rm i.]
\item a $K_{d(v_i)}$ with
$m_i$ trivial cliques $\{v_iv_i^1\},\ldots ,\{v_iv_i^{m_i}\}$, or
\item an $\NP$ or a $\PP$ of $K_{d(v_i)}$.
\end{enumerate}
Furthermore, if $d(v_i)=m_i+1$ and the cliques in $\mathcal{P}$
induced by $v_i$-stars constitute an \mbox{\textup{N-P}}, then there
are two non-isomorphic $\mathcal{S}$ which are determined according
to the intersection of the $K_2$'s in the \mbox{\textup{N-P}} is
$v_iv_i^{m_i+1}$ or not unless $m_i = 2$. This proves the last
equality in the theorem. \qed\end{proof}

\begin{theorem} \label{linegraph m.s.d.u.-rep.}
For a graph $G$,
\begin{equation*}
\tau_\mathrm{sdu}(G^*)=
\begin{cases}
2 & \text{if $G \in\{K_4, W_2, \mathrm{TP}_1\}$ with $m_1=1$,}\\
N_{\PP}(d(v),r) &
\text{if $G$ is $H$ with $\theta_{sdu}(K_{d(v)})=d(v)$,}\\
1+N_{\PP}(d(v)+1,r) &
\text{if $G$ is $H$ with $\theta_{sdu}(K_{d(v)})=d(v)+1$,}\\
1 & \mbox{otherwise,}
\end{cases}
\end{equation*}
where $H$ is a $v$-star with $d(v)\geqslant 4$
\end{theorem}
\begin{proof}
If $G$ is neither $K_3$ nor a star, then $G^*$ is not a complete
graph. Therefore, any $\mathcal{S}\in F_\mathrm{s}(G^*)$ has
$|S(e)|\geq 2$ for some $e\in E(G)$. This further implies that any
$\mathcal{S}\in F_{su}(G^*)$ has $|S(e)|\geq 2$ for {\em all\/}
$e\in E(G)$.

\medskip

\noindent Any pair of edges $e$ and $e'$ with $e\neq e'$ have
$S(e)\nsubseteq S(e')$ since $|S(e)\cap S(e')|\leq 1$. Thus those
members, if any, of $F_{sa}(G^*)$ in which all sets have the same
cardinality, constitute $F_{su}(G^*)$ and $F_{sdu}(G^*)$. For the
case where $G$ is a $v$-star, we can obtain $\tau_{sdu}(G^*)$
directly from Theorem~\ref{SDUomega K_n}. This completes the proof.
\qed\end{proof}

\end{document}